\documentclass[a4paper,11pt]{article}
\usepackage[applemac]{inputenc}
\usepackage[T1]{fontenc}

\usepackage{lmodern}
\usepackage{array}
\usepackage{graphicx}
\def\shrug{\texttt{\raisebox{0.75em}{\char`\_}\char`\\\char`\_\kern-0.5ex(\kern-0.25ex\raisebox{0.25ex}{\rotatebox{45}{\raisebox{-.75ex}"\kern-1.5ex\rotatebox{-90})}}\kern-0.5ex)\kern-0.5ex\char`\_/\raisebox{0.75em}{\char`\_}}}
\usepackage{amssymb, amsfonts, amsmath, amsthm, bbm}
\usepackage{enumerate}
\usepackage[bottom]{footmisc}

\usepackage[english]{babel}

\usepackage[colorlinks]{hyperref}

\usepackage{tikz}
\usepackage{subcaption}
\usepackage[font={small,it}]{caption}

\numberwithin{equation}{section}

\theoremstyle{plain}
\newtheorem{theorem}{Theorem}[section]
\newtheorem{corollary}[theorem]{Corollary}

\newtheorem{lemma}[theorem]{Lemma}
\newtheorem{proposition}[theorem]{Proposition}

\theoremstyle{definition}

\theoremstyle{remark}
\newtheorem{remark}[theorem]{Remark}

\newcommand{\N}{\mathbb{N}}

\newcommand{\R}{\mathbb{R}}

\renewcommand{\S}{\mathbb{S}}
\newcommand{\D}{\mathbb{D}}

\newcommand{\ind}[1]{\mathbbm{1}_{\left\{#1\right\}}}

\DeclareMathOperator{\E}{\mathbb{E}}
\renewcommand{\P}{\mathbb{P}}

\newcommand{\dd}{\mathrm{d}}

\renewcommand{\bar}[1]{\overline{#1}}
\renewcommand{\tilde}[1]{\widetilde{#1}}

\renewcommand{\epsilon}{\varepsilon}

\renewcommand{\H}{\mathbb{H}}

\usepackage{xcolor}

\usepackage{empheq}
\usepackage[framemethod=tikz]{mdframed}
\usepackage{lipsum}
\usepackage{tcolorbox}
\usepackage{fourier-orns}
\usepackage{enumitem}

\numberwithin{equation}{section}

\newcommand{\apropto}{\mathrel{\vcenter{
  \offinterlineskip\halign{\hfil$##$\cr
    \propto\cr\noalign{\kern2pt}\sim\cr\noalign{\kern-2pt}}}}}

\newcommand{\h}{\mathrm{h}}

\usepackage[english]{babel}

\date{}

\title{Fragmentation processes and the convex hull of the Brownian motion in a disk}

\author{B\'en\'edicte Haas \thanks{Universit\'e Sorbonne Paris Nord, LAGA, CNRS (UMR  7539) 93430 Villetaneuse, France \newline \hspace*{0.5cm} E-mail: haas@math.univ-paris13.fr}  \quad \& \hspace{0.2cm}  Bastien Mallein\thanks{Institut de Mathématiques de Toulouse, UMR 5219 - Université de Toulouse,  France  \newline \hspace*{0.5cm} E-mail: bastien.mallein@math.univ-toulouse.fr}}

\begin{document}

\let\oldproofname=\proofname
\renewcommand{\proofname}{\rm\bf{\oldproofname}}

\maketitle

\begin{abstract}
Motivated by the study of the convex hull of the trajectory of a Brownian motion in the unit disk reflected orthogonally at its boundary, we study inhomogeneous fragmentation processes in which particles of mass $m \in (0,1)$ split at a rate proportional to $|\log m|^{-1}$. These processes do not belong to the well-studied family of self-similar fragmentation processes. Our main results characterize the Laplace transform of the typical fragment of such a process, at any time, and its large time behavior.

We connect this asymptotic behavior to the prediction obtained by physicists in \cite{DBBM22} for the growth of the perimeter of the convex hull of a Brownian motion in the disc reflected at its boundary. We also describe the large time asymptotic behavior of the whole fragmentation process. In order to implement our results, we make a detailed study of a time-changed subordinator, which may be of independent interest.
\end{abstract}

\section{Introduction and main results}

\subsection{On the convex hull of the Brownian motion in the plane and the disc}

Consider a Brownian particle $B$ in the plane, which might be thought of as the trajectory of an animal exploring its territory foraging for food. A natural way to estimate the area covered by this animal during its search for $t$ units of time is by estimating the convex hull $\tilde{H}_t = \mathrm{Hull}(\{B_s, s \leq t\})$ of its position. Using Cauchy' surface area formula \cite{Cau32,TsV}, the length $\tilde{P}_t$ of the perimeter of the convex set $\tilde{H}_t$ can be computed as
\[
  \tilde{P}_t = \int_{0}^{2\pi} \sup_{0 \leq s \leq t} (e_\theta \cdot B_s) \dd \theta
\]
writing $e_\theta = (\cos(\theta),\sin(\theta))$. Using the invariance by rotation of the law of $B$, it is then a simple exercise (see \cite{Let} and the references therein) to show that in this case
\[
  \E\Big[\tilde{P}_t \Big] = 2\pi \E\Big[ \max_{s \leq t} (e_0\cdot B_s) \Big] = 2\pi \E\Big[|e_0 \cdot B_t|\Big] = \sqrt{8\pi t}
\]
applying the reflection principle for the unidimensional Brownian motion $e_0 \cdot B$.

In \cite{DBBM22}, De Bruyne, B\'enichou, Majumdar and Schehr take interest in a similar problem associated to the Brownian motion confined to the unit disk $\D$. Write $B^\D$ for a standard Brownian motion in $\D$, starting from 0, with orthogonal reflection at the boundaries. To estimate the area covered by the process, they study the convex hull $\bar{H}_t = \mathrm{Hull}(\{B^\D_s, s \leq t\})$. As the process $B^\D$ is recurrent in $\D$, it is natural to expect that $\bar{H}_t$ converges to $\D$ as $t \to \infty$, hence that $\bar{P}_t$ the length of the perimeter of $\bar{H}_t$ converges to $2\pi$ (see Figure~\ref{fig:brownianHull}).

\begin{figure}[ht]
\centering
\begin{subfigure}{.23\linewidth}
\includegraphics[width=\textwidth]{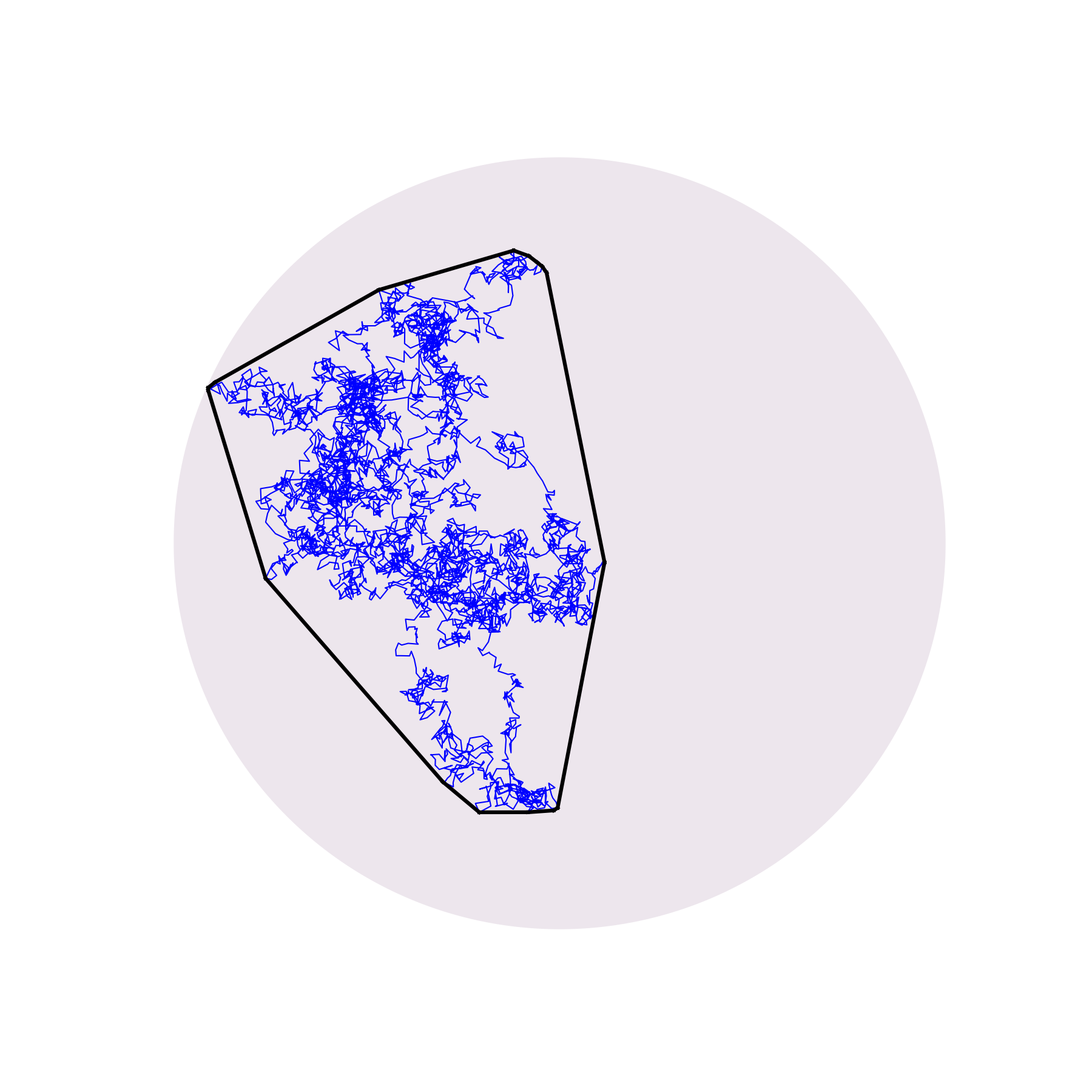}
\caption{Time $t = 1$}
\end{subfigure}
\begin{subfigure}{.23\linewidth}
\includegraphics[width=\textwidth]{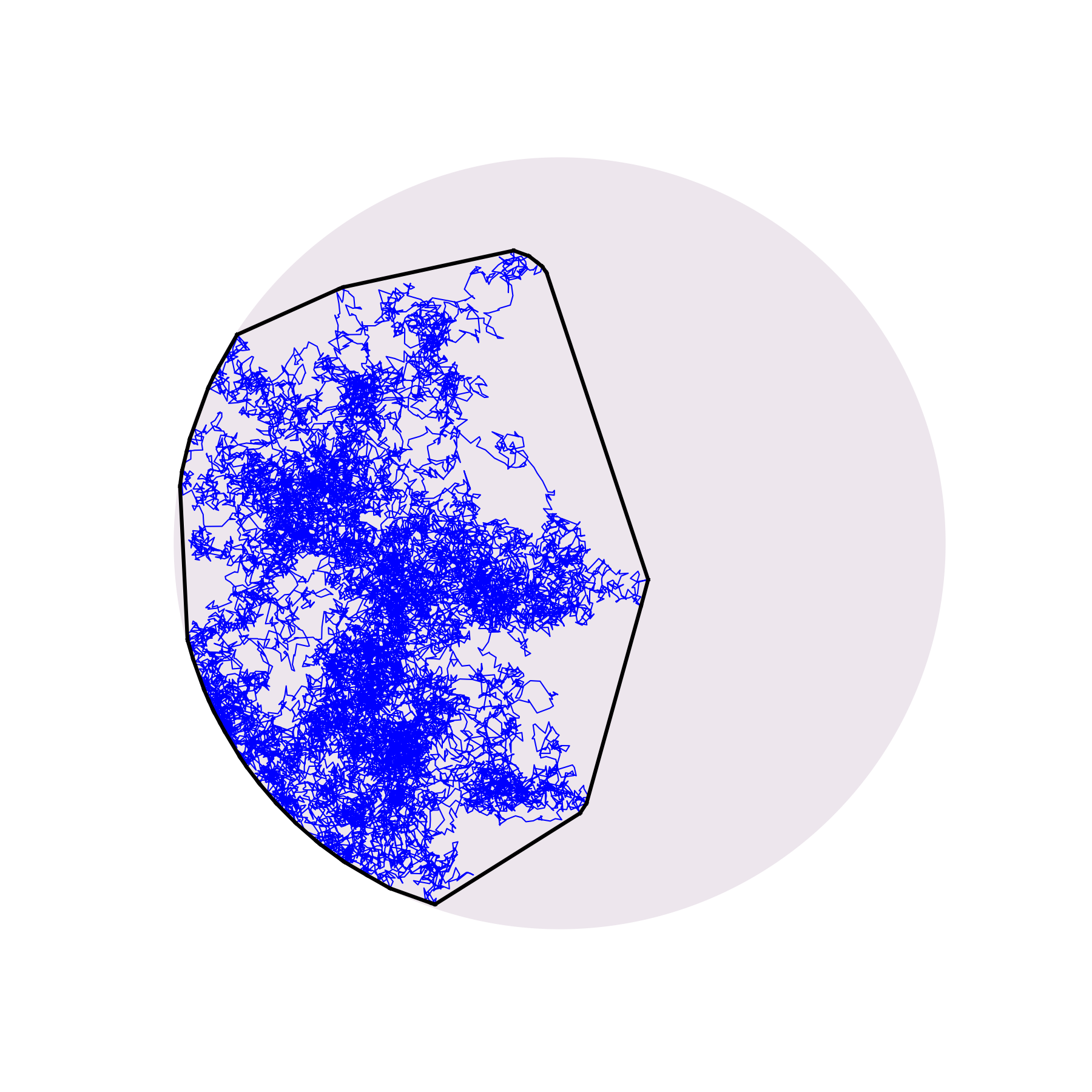}
\caption{Time $t = 4$}
\end{subfigure}
\begin{subfigure}{.23\linewidth}
\includegraphics[width=\textwidth]{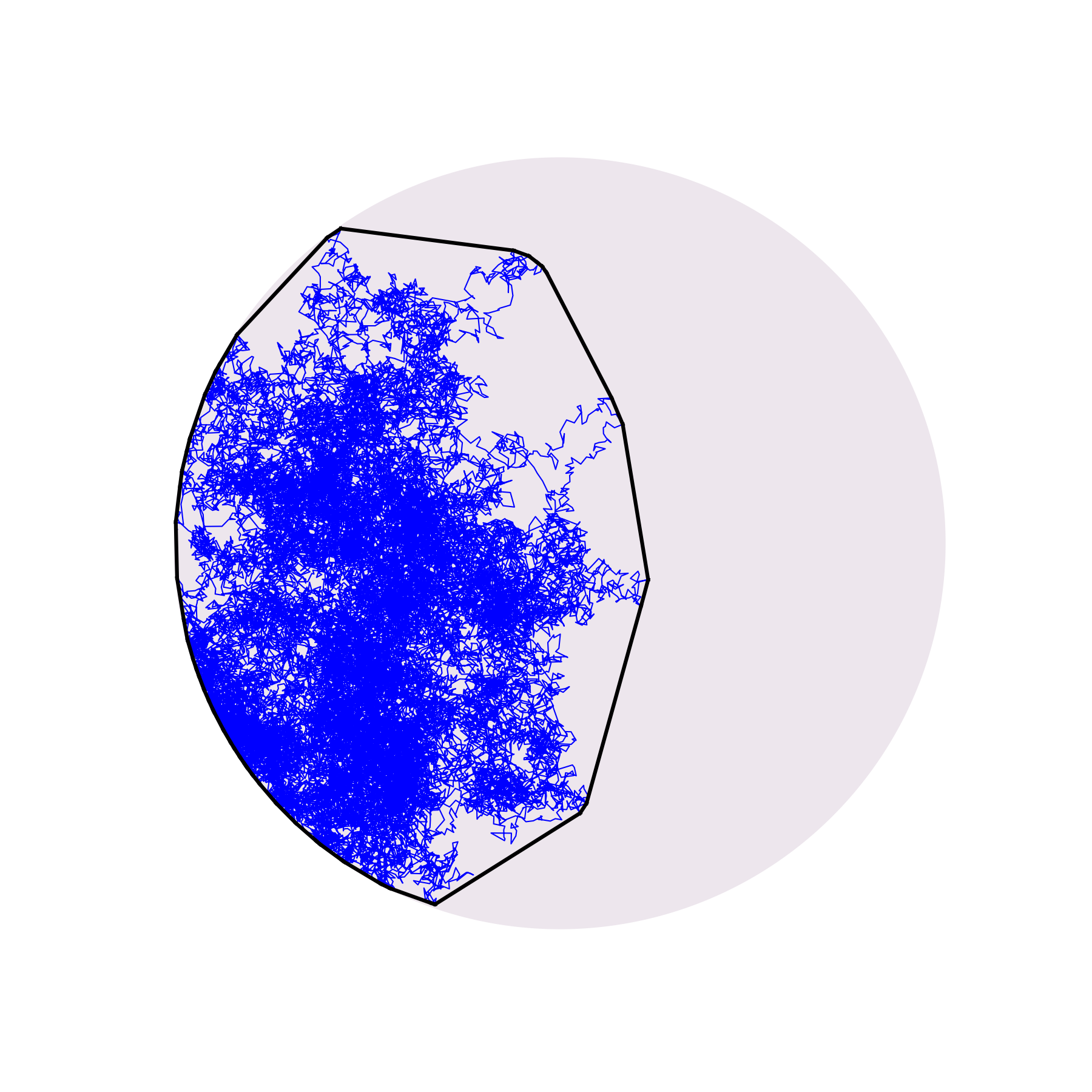}
\caption{Time $t = 8$}
\end{subfigure}
\begin{subfigure}{.23\linewidth}
\includegraphics[width=\textwidth]{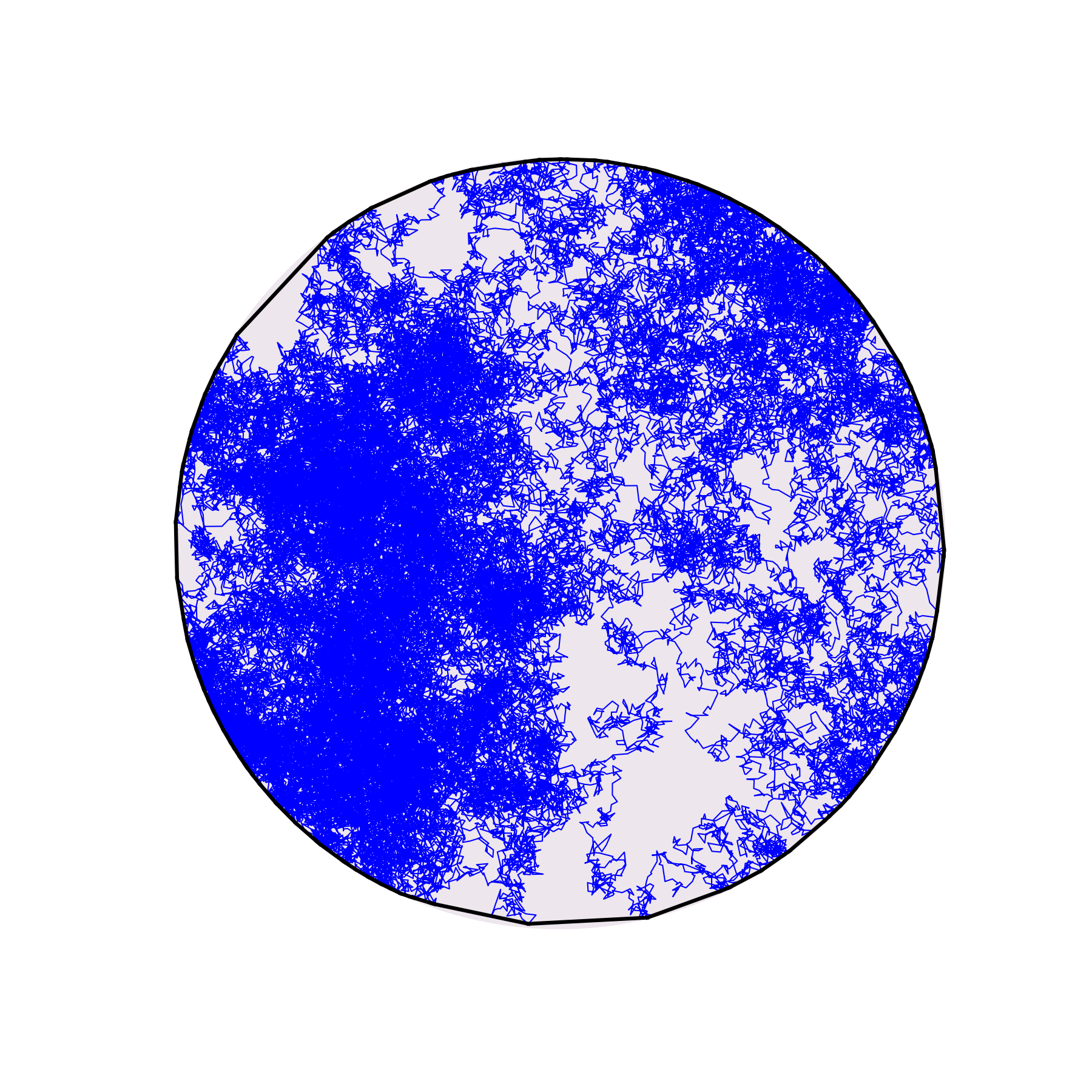}
\caption{Time $t = 16$}
\end{subfigure}
\caption{Convex hull (in black) of the trajectory of the Brownian motion $B^\D$ (in blue) in the unit disk over the time interval $[0,t]$.}
\label{fig:brownianHull}
\end{figure}

Using again the Cauchy formula and the invariance by rotation of the process, De Bruyne et al. \cite{DBBM22} obtain the following exact equality for the average length of the convex hull of $\{B^\D_s, s \leq t\}$:
\begin{equation*}
  \E\Big[\bar{P}_t\Big] = 2 \pi \E\Big[\sup_{s \leq t}(e_0\cdot B^\D_s)\Big] = 2 \pi \int_0^1 \P(\tau_x < t) \dd x,
\end{equation*}
with $\tau_x = \inf\{t > 0: B^\D_t\cdot e_0 > x\}$ being the first entrance time of $B^\D$ inside the spherical cap $\{z \in \D : z .e_0>x\}$. Estimating $\P(\tau_x > t)$ as $x \to 1$ and $t \to \infty$ is related to the well-known \emph{narrow escape problem} \cite{narrowEscapeProblem,Rupprecht2014}. Using the empirical approximation
\begin{equation}
  \label{eqn:narrowEscapeApproximation}
  \P(\tau_x > t) \approx \exp\left( - \frac{t}{-\log(1-x) -\log 2 + \frac{1}{2}} \right) \quad \text{ as $t \to \infty$}
\end{equation}
for all $x$ small enough, justified in \cite[Section 2.3]{DBBM22}, De Bruyne et al. then make the following prediction for the asymptotic behavior of $\E\big[\bar{P}_t\big]$:
\begin{equation}
\label{equi:DBBM22}
  2\pi - \E\big[\bar{P}_t\big] \underset{t \to \infty}{\sim} c \cdot t^{1/4}e^{-2 t^{1/2}}
\end{equation}
for some $c > 0$.

We present in this article a toy-model for the evolution of the convex hull of the Brownian motion in the disk based on time-inhomogeneous fragmentation processes, which present a similar asymptotic behavior as $t \to \infty$. These processes do not belong to the well-studied family of self-similar fragmentation processes whose study was initiated in \cite{BertoinSSF} and, interestingly, we still manage to obtain precise results, both at large and fixed times. 
From the Brownian in the disk point of view, in addition to the asymptotic of the mean of the perimeter, this toy model also allows us to study in depth the distribution of the length of a typical face of the convex hull, as well as the empirical distribution of the whole set of lengths of its faces.

\subsection{An inhomogeneous fragmentation approximation}
\label{sec:intro_frag}

\textbf{Heuristics.} Let us first observe that the asymptotic behavior of $\bar{P}_t$ can heuristically be approached by a fragmentation process. Indeed, let $C_t = \{B^\D_s, s \in [0,t]\} \cap \S^1$ be the set of positions at which the Brownian particle hits the boundary of $\D$. We denote by $(\ell_j(t), j \geq 1)$ the lengths of the intervals in $\S^{1} \setminus C_t$, ranked in the non-increasing order. Writing $P_t$ and $A_t$ for the length of the perimeter and the area of the convex hull of $C_t$ respectively, we observe that
\begin{equation}
\begin{split}
  2 \pi - P_t &= \sum_{i \geq 1} \ell_i(t) - 2 \sin(\ell_i(t)/2) \underset{t \to \infty}{\sim}  \frac{1}{24} \sum_{i \geq 1} \ell_i(t)^3 \label{eqn:perimeter}\\
  \pi - A_t &= \sum_{i \geq 1} \frac{\ell_i(t)}{2} - \sin(\ell_i(t)/2) \underset{t \to \infty}{\sim}  \frac{1}{12} \sum_{i \geq 1} \ell_i(t)^3.
  \end{split}
\end{equation}
We expect $\bar{H}_t$ and $\mathrm{Hull}(C_t)$ to be very close sets as $t \to \infty$, see Figure~\ref{fig:approximation}. Indeed, one immediately remark that
$\mathrm{Hull}(C_t) \subset{\bar{H}_t}$ and $P_t \leq \bar P_t$.
Using again Cauchy's formula and invariance by rotation, we note that
\begin{equation*}
  \E\big[P_t\big] = 2 \pi \E\Big[\max_{z \in C_t} z \cdot e_0\Big] = 2 \pi \int_0^1 \P(\tau^{\mathrm{ne}}_x < t) \dd x,
\end{equation*}
where $ \tau^\mathrm{ne}_x = \inf\{t > 0 : B^\D_t \in \S^1 \text{ and } B^\D_t \cdot e_0 > x \}$ is the stopping time associated to the narrow escape problem from a target of length $2 \arccos(x) \sim 2 \sqrt{2(1-x)}$ as $x \to 1$. Using that the boundary of $\D$ is locally well-approximated by its cord, it is argued in \cite[Section 2.1]{DBBM22} that for $x$ small enough,
\[
  \P(\tau^\mathrm{ne}_x > t) \approx \P(\tau_x > t) \text{ as $t \to \infty$}.
\]
Numerical simulations supporting that claim are given in \cite[Appendix A]{DBBM22}. As a result, we expect that $\lim_{t \to \infty} \frac{\E[\bar{P}_t] - \E[P_t]}{2\pi - \E[\bar{P}_t]} = 0$. Simulations of our own, drawn in Figure~\ref{fig:approximation}, give credit to this approximation.

\begin{figure}[ht]
\centering
\begin{subfigure}{.3\linewidth}
\includegraphics[width=\textwidth]{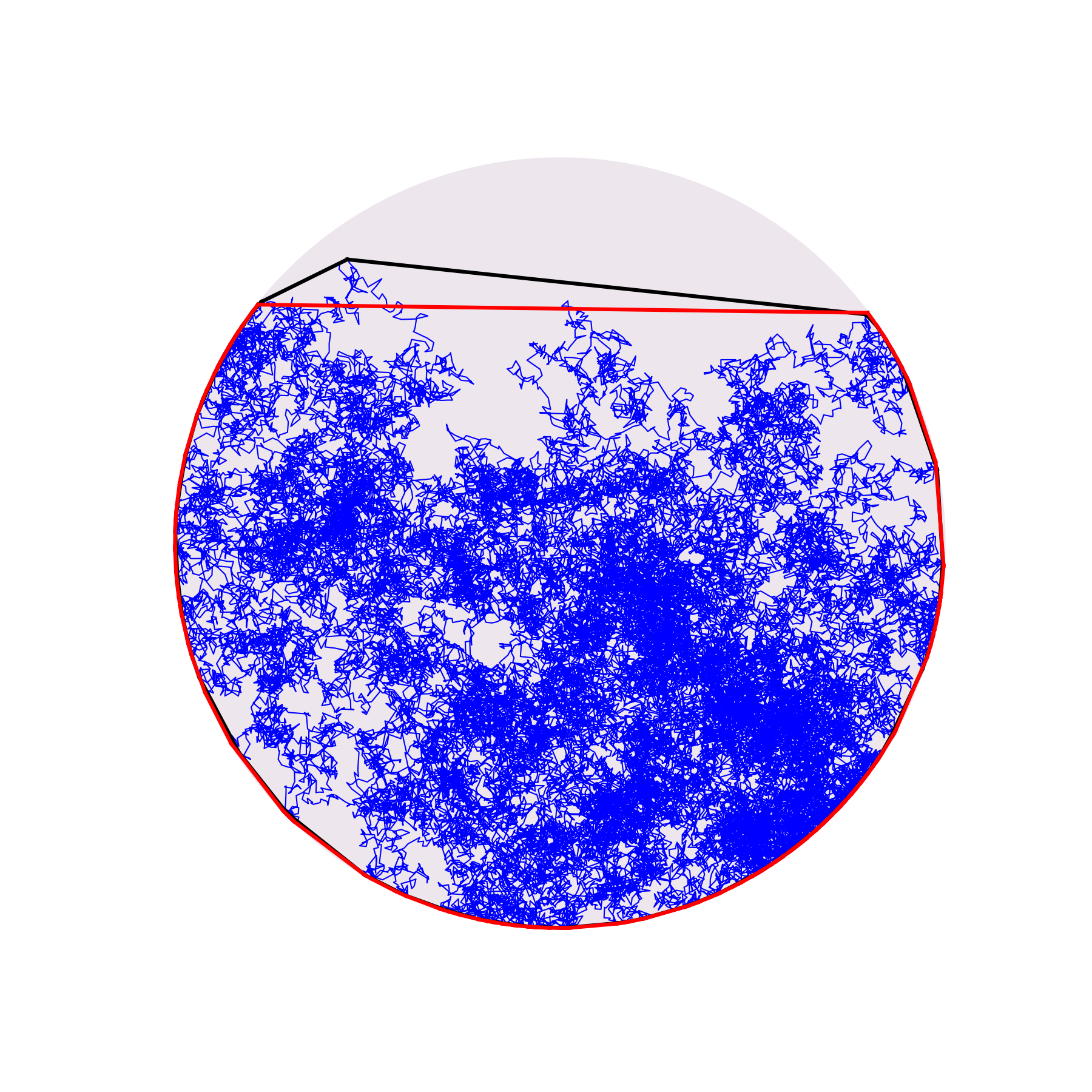}
\end{subfigure}
\hspace{1.5cm}
\begin{subfigure}{.45\linewidth}
\includegraphics[width=\textwidth]{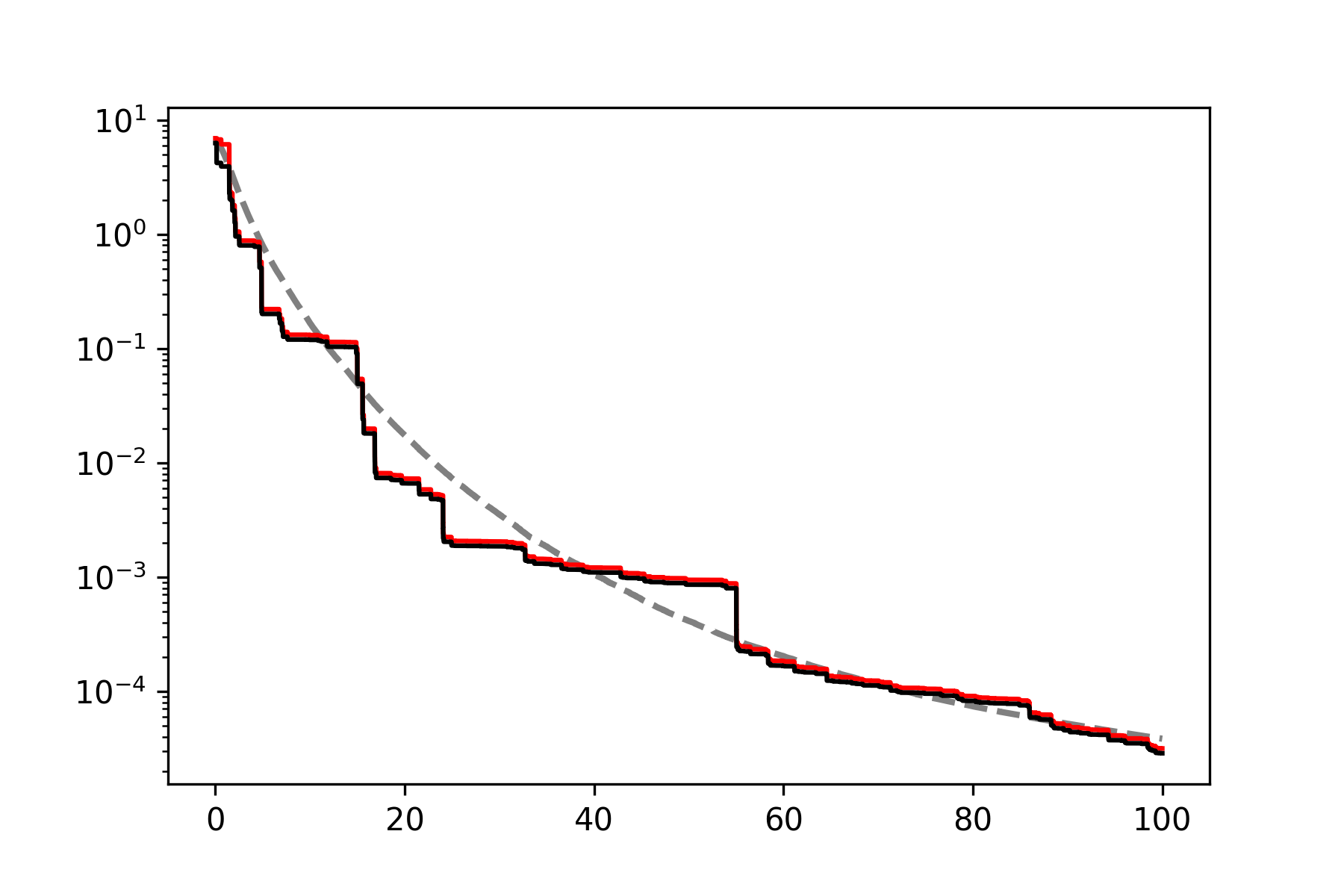}
\end{subfigure}
\caption{On the left-hand side: the convex hulls $\bar{H}_t$ (in black) and $\mathrm{Hull}(C_t)$ (in red) of the trajectory of the Brownian motion $B^\D$ and of the set of points at which the trajectory hits the boundary of $\D$ respectively. As expected, the boundary of the two hulls become indistinguishable around points at which the Brownian curve hits the boundary of $\D$. \hspace{0.02cm} On the right-hand side: the evolution over time of $2 \pi - \bar{P}_t$ (in black) and $2\pi - P_t$ (in red). As expected, after a short transitory period, the two curves become indistinguishable, even on a log scale. We additionally drew in grey an approximation of $2\pi - \E[P_t]$ computed using Monte-Carlo methods, which shows a decay of order $e^{-c\sqrt{t}(1+o(1))}$.}
\label{fig:approximation}
\end{figure}

We now observe that $(\ell_i(t), i \geq 1)_{t \geq 0}$ behaves as a generalized fragmentation process. Indeed, in this process, at random times intervals get cut by the trajectory of $B^\D$. The time at which a fragment of length $\ell$ gets split has the law of $T_\ell$ the first hitting time by $B^\D$ of a target of width $\ell$ on $\S^1$. Obviously, the law of this random variable depends heavily on the starting position of $B^\D$. However, we have $\max_i \ell_i(t) \to 0$ as $t \to \infty$, using the recurrence of $B^\D$. We therefore argue that as $t \to \infty$, the splitting times of different intervals should become essentially independent of one another, and the dependency on the starting position should become less relevant. To be more precise, assuming that $B^\D$ starts from the uniform distribution in $\D$, we recall that (see e.g. \cite{narrowEscapeProblem,Rupprecht2014})
\[
  \E\left[T_\ell\right] \underset{\ell \to 0}{\sim} - \log \ell,
\]
and the narrow escape approximation  \eqref{eqn:narrowEscapeApproximation} consists in considering that as $\ell \to 0$, $T_\ell$ approaches an exponential distribution with parameter $-(\log \ell)^{-1}$. Additionally, the scaling properties of the Brownian motion implies that once $B^\D$ hits this interval, it will split it into smaller fragments in a self-similar fashion, similarly to a $2$-dimensional Brownian motion on the half-plane hitting a domain of length $\ell$ on its boundary.

In order to approach the asymptotic behavior of the process $(\ell_i(t)/2\pi, i \geq 1)_{t \geq 0}$, we therefore propose to study inhomogeneous fragmentation processes in which distinct particles evolve independently and split in a similar fashion, with the constraint that  the rate of splitting of a particle with mass $m \in (0,1)$ is proportional to $|\log(m)|^{-1}$. We now introduce such processes rigorously.

\bigskip

\noindent \textbf{Fragmentation processes $(\nu,c,|\log|^{-1})$.} We place ourselves within the framework of the theory of fragmentation processes developed by Bertoin in the early 2000s \cite{BertoinHom, BertoinSSF, BertoinAB03}.
Let $\nu$ be a (possibly infinite) Radon measure on the set of (possibly improper) partition masses of the unit interval
\[
  \mathcal{S} = \left\{ \mathbf{s} = \left(s_i, i \geq 1\right) \in [0,1]^\N  : s_1 \geq s_2 \geq \cdots, \sum_{i \geq 1} s_i \leq 1  \right\}
\]
such that
\begin{equation}
  \label{eqn:integrabilityAssumption}
  \int_{\mathcal{S} }(1- s_1) \nu(\dd \mathbf{s}) < \infty.
\end{equation}

Before introducing \emph{inhomogeneous fragmentation processes}, where the rates of splitting depend on the masses of the particles, we first roughly recall Bertoin's construction of  \emph{homogeneous fragmentation processes} where particles of mass $m$ get fragmented into particles of mass $ms_1,ms_2,\ldots$ at the common rate $\nu(\dd s)$, whatever the mass $m$ is. Such a process is defined as a random family of pairwise disjoint open subintervals $(I^\h_i(t), i \in \N)$ of $(0,1)$, where here $t$ denotes the time. Let $N$ be a Poisson point process on $\mathcal{S} \times \N \times \R_+$ with intensity $\nu(\dd \mathbf{s}) \otimes\#(\dd k)  \otimes\dd t$ and start at time $0$ with $I^\h_1(0) = (0,1)$. Then for each atom $(\mathbf{s},k,t)$ of $N$, the interval $I^\h_{k}(t-)$ is fragmented at time $t$ into subintervals of length $|I^\h_k(t-)|s_1,|I^\h_k(t-)|s_2,\cdots$. The intervals $(I^\h_i(t), j \geq 1)$ are then relabelled in decreasing order of their lengths.  The well-definition of this construction is guaranteed by assumption \eqref{eqn:integrabilityAssumption}. Additionally, the fragments can melt according to a parameter $c \geq 0$. We refer to \cite{BertoinHom, BertoinSSF} for details and call here the process $(|I^\h_i(t)|, i \geq 1)_{t \geq 0}$ \emph{a homogeneous fragmentation process with dislocation measure $\nu$ and erosion coefficient $c$}.

For each $x \in [0,1]$, we denote by $I^\h_{t,x}$ the unique interval in $(I^\h_j(t), j \geq 1)$ that contains $x$ if it exists and set $I^\h_{t,x}:=\emptyset$ otherwise. Let $\tau:(0,1] \to (0,\infty)$ be a continuous function. An \emph{inhomogeneous fragmentation process} in which particles of mass $m$ split at rate $\tau(m) \nu(\mathrm ds)$ and melt continuously at rate $c$ can be constructed from the above homogeneous fragmentation by using a Lamperti-type time change. More precisely, for all $x \in (0,1)$ and all $t\geq 0$, we set
\begin{equation}
\label{time_change_tau}
  T^x_t = \inf\left\{u > 0 : \int_0^u \frac{1}{\tau(I^\h_{r,x})}\dd r >t\right\}.
\end{equation}
We then consider the family\footnote{Observe that if $y \in I^\h_{T^x_t}$, then $T^y_t = T^x_t$, so this family indeed consists in pairwise disjoint subsets of $(0,1)$.} $\cup_{x \in (0,1)} \{I^\h_{T^x_t,x}\}$, and denote by $(I_i(t), i \geq 1)$ this family of intervals ranked in the decreasing order of their lengths. The inhomogeneous fragmentation process with parameters $(\nu,c,\tau)$ is then given by $(|I_i(t)|, i \geq 1)_{t \geq 0}$. Again we refer to \cite{BertoinSSF} for details, in particular when the functions $\tau$ are power functions, which then lead to \emph{self-similar fragmentation processes}, a family of processes that have been intensely studied. See also \cite{H03} for an extension to more general functions $\tau$. Here we apply this construction to $\tau(x)=|\log(x)|^{-1}$. We note that then $\tau(1)=\infty$, which takes us a bit out of the previous framework, but that does not prevent us from doing the construction. It simply implies that $I_1(0)=I^{\h}_1(T_1), i\geq 1$ when $c=0$, $\nu$ is finite and $T_1$ denotes the first jump time of $I^\h$, and that $I_1(0)=(0,1)$ otherwise.

From now on, we focus on such inhomogeneous fragmentation processes with parameters $$(\nu, c,|\log |^{-1}).$$ For $t \geq 0$ and $i \geq 1$, we write $F_i(t)$ for the size at time $t$ of the $i$th largest element of the fragmentation. We wish to study the properties of such a process, starting with the law of its typical fragment, defined as the trajectory of the process $t \mapsto |I_{T^{x}_t,U}|$, where $U$ is uniformly sampled in $(0,1)$. Following \cite{BertoinSSF}, \cite{H03}, there is the following classical representation of this trajectory.
\begin{proposition} [\cite{BertoinSSF}, \cite{H03}]
\label{prop:taggedFragment}
Let $\xi^{(\nu,c)}$ be a subordinator with Laplace exponent
\[
\phi_{(\nu,c)}(q)= - \log \E\Big[e^{-q \xi^{(\nu,c)}_1}\Big] := cq+\int_{\mathcal{S}} \sum_{i\geq 1}s_i(1-s_i^{q})\nu(\dd \mathbf{s}), \quad q\geq 0.
\]
Let $\rho(t) := \inf \left\{ u \geq 0 : \int_0^u \xi^{(\nu,c)}_r \dd r \geq t \right\}$, $t\geq 0$, with the convention $\inf\{\emptyset\}=\infty$. Then for any measurable bounded function $f:[0,1] \rightarrow \mathbb R$, one has
\[
  \E\left[ f\Big(e^{-\xi^{(\nu,c)}_{\rho(t)}}\Big) \right] =  \E\Bigg[\sum_{i \geq 1} F_i(t) f\left(F_i(t)\right) \Bigg].
\]
\end{proposition}

In other words, the law of a typical fragment at time $t$, chosen proportionally to its weight, is distributed as $e^{-\xi^{(\nu,c)}_{\rho(t)}}$. Recalling the estimate \eqref{eqn:perimeter}, we observe that for all $q \geq 0$
\begin{equation}
\label{eq:Laplace_sum}
  \E\Bigg[  \sum_{i \geq 1} F_i(t)^{q+1}\Bigg] = \E\left[ e^{-q \xi^{(\nu,c)}_{\rho(t)}} \right].
\end{equation}
One of our objectives is therefore to study the Laplace transforms of the time-changed subordinator $\xi^{(\nu,c)}_{\rho(t)}$, at any time $t$. We will undertake this study for all subordinators, and not necessarily those having a Laplace exponent of the form $\phi_{(\nu,c)}$.

\subsection{Main results}
\label{sec:def_fragmentations}

Motivated by the previous discussion, we consider now a generic subordinator $\xi$, with Laplace exponent
\[
 \phi(q)=-\log \E\left[e^{-q \xi_1}\right] = \kappa + c q + \int_0^\infty (1 - e^{-qx}) \pi(\dd x)
\]
where $\kappa \geq 0$, $c \geq 0$ and $\pi$ is a measure on $(0,\infty)$ such that $\int_0^{\infty} (1 \wedge x) \pi(\dd x) < \infty$. We refer to $\pi$ as the L\'evy measure of $\xi$, $c$ its drift and $\kappa$ its death rate. We define then the time change $\rho$ by
\[
  \rho(t) := \inf \left\{ r \geq 0 : \int_0^r \xi_u \dd u \geq t  \right\}, \quad t \geq 0,
\]
with the convention $\inf{\emptyset}=\infty$. Our first main result gives the following exact expression for the Laplace transform of $\xi_{\rho(t)}$, for any $t\geq 0$.

\medskip

\begin{theorem}
\label{thm:lapTransform}
Let $\Phi(q)=\int_0^q \phi(s) \dd s$, $q \geq 0$ and $\Phi^{-1}: [0,\infty) \mapsto  [0,\infty)$ be the inverse of $\Phi$. Then the function $\Phi^{-1}$ is the Laplace exponent of a subordinator, with a Lévy measure that we denote $L$, and for all $q > 0$ and $t \geq 0$, we have
\begin{equation}
  \label{eq:Laplace_intro}
  \mathbb E\left[e^{-q\xi_{\rho(t)}}\right]   =    \phi(q)\int_0^{\infty} e^{-\Phi(q)x-\frac{t}{x}} xL(\mathrm dx).
\end{equation}
\end{theorem}

This result relies on a surprising connection between the law of $\xi_{\rho(t)}$ and a spectrally negative Lévy process $X$ with Laplace exponent $\Phi$. This will be discussed in Section~\ref{sec:Laplace}. Although the L\'evy measure $L$ is generally quite abstract, this result allows us to obtain precise estimates on $ \mathbb E\big[e^{-q\xi_{\rho(t)}} \big]$. We first obtain an upper bound for $ \mathbb E\big[e^{-q\xi_{\rho(t)}} \big]$ which is not sharp but is valid for any subordinator $\xi$:

\begin{proposition}
\label{prop:unifBound}
For all $q > 0$, there exists $a(q) \in (0,\infty)$  such that for all $t \geq 0$
\[
\E\left[ e^{-q \xi_{\rho(t)}} \right]   \leq   a(q) \cdot  \left(1+t^{1/8}\right) e^{-2 \sqrt{\Phi(q)t}}.
\]
\end{proposition}
In particular, this shows with the relation (\ref{eq:Laplace_sum}) and the estimate \eqref{eqn:perimeter},  that our toy-model can capture the exponential decay rate found in \eqref{equi:DBBM22}, but the prefactor $t^{1/4}$ cannot be recovered from such a simple model. We indicate in the final Section~\ref{sec:discussions} some possible causes for this discrepancy.

\medskip

\begin{remark}
Observe that when $\kappa>0$, $\xi_{\rho(t)}$ may be infinite. Letting $q \downarrow 0$ in (\ref{eq:Laplace_intro}), we see that
\[
\mathbb P\big(\xi_{\rho(t)}<\infty\big)=\kappa \int_0^{\infty}e^{-\frac{t}{x}} xL(\mathrm dx), \quad \text{ for all } t \geq 0
\]
(we will see further that the measure $xL(\mathrm dx)$ is finite if and only if $\kappa>0$).
\end{remark}

\medskip

Under some more precise regularity conditions, we can compute the asymptotic behavior of the Laplace transform of $\xi_{\rho(t)}$ as $t \to \infty$. We first assume that
\begin{equation}
  \label{hyp:rv}
  \tag{$\mathbf {H_{\gamma}}$}
  \phi \text{ is regularly varying at 0 with index }\gamma \in (0,1],
\end{equation}
i.e. that the subordinator $\xi$ belongs to the domain of attraction of a stable random variable with index $\gamma$. We recall that a function $f$ is called regularly varying at $0$ with index $\alpha \in \mathbb R$ if for all $\lambda > 0$,
\[
  \lim_{x \to 0} \frac{f(\lambda x)}{f(x)} = \lambda^\alpha,
\]
and refer to the book of Bingham, Goldie and Teugels \cite{BGT} for background on regularly varying functions. In particular, observe that under assumption \eqref{hyp:rv}, we have $\kappa = 0$. In addition to the regular variation of $\phi$, we require an extra technical assumption, which guarantees that the law of $X_1$ is strongly non-lattice, namely that the three following conditions hold
\begin{equation}
\label{hyp:tech}
\tag{$\mathbf {H_{\mathrm{tech}}}$}
\int_0^{\infty}\frac{\pi(\mathrm dx)}{x}=\infty, \quad \int_{1}^{\infty}e^{a\mathrm{Re}(\Phi(ix))}\mathrm dx<\infty \text{ for some }a>0, \quad \limsup_{x \rightarrow \infty} \mathrm{Re}(\Phi(ix)) <0.
\end{equation}
This condition comes from our use of results by Doney and Rivero \cite{DR13,DR16_err}, but we are not convinced that they are necessary for the following result to hold.

\begin{theorem}
\label{thm:asympBehavior}
Assume \eqref{hyp:rv} and \eqref{hyp:tech} and define for all $q > 0$
\begin{equation*}
  \label{eqn:defcq}
  b(q) := \frac{\sqrt{\pi}}{(\gamma+1)\Gamma(\frac{\gamma}{\gamma + 1})} \cdot \phi(q) \Phi(q)^{\frac{1}{2(\gamma + 1)} - \frac{3}{4}}.
\end{equation*}
Then we have
\[
  \E\left[ e^{-q \xi_{\rho(t)}}\right]   \underset{t \to \infty} \sim   b(q)\cdot  t^{1/4} \Phi^{-1}(t^{-1/2})  e^{-2 \sqrt{\Phi(q)t}}.
\]
\end{theorem}

\medskip

\begin{remark}
Note that under assumption \eqref{hyp:rv}, $\Phi$ and $\Phi^{-1}$ are both regularly varying at $0$ with respective indices $\gamma + 1 \in (1,2]$ and $\frac{1}{\gamma+1} \in [1/2,1)$. Therefore, Theorem~\ref{thm:asympBehavior} implies that $e^{\sqrt{2\Phi(q)t}} \E\big[e^{-q \xi_{\rho(t)}}\big]$ is regularly varying at $\infty$ with index $\frac{1}{4}-\frac{1}{2(\gamma+1)} \in (-\frac{1}{4},0]$. It hints at the fact that in the bound in Proposition~\ref{prop:unifBound} is not optimal.
\end{remark}

Letting next $q$ depends on $t$ in (\ref{eq:Laplace_intro}),  we further obtain the large time scaling limit of $\xi_{\rho}$. To describe it, let $D(\gamma)$ denote the distribution on $(0,\infty)$ with Laplace transform
$$
q \in (0,\infty) \mapsto \frac{q^{\gamma}}{\Gamma\left(\frac{\gamma}{\gamma+1} \right)} \int_0^{\infty} e^{-q^{\gamma+1}u-\frac{1}{u}} \frac{ \mathrm du}{u^{\frac{1}{1+\gamma}}}.$$
This distribution will be introduced in Section~\ref{sec:stable} as the stationary distribution of the process $\Big(t^{-1/(\gamma+1)}\xi^{(\gamma)}_{\rho(t)}, t\geq 0\Big)$, where $\xi^{(\gamma)}$ is a $\gamma$-stable subordinator.

\smallskip

\begin{theorem}
\label{thm:asymp_distribution}
\begin{enumerate}
\item[\emph{(i)}]
Under assumption \eqref{hyp:rv},
$$
  \Phi^{-1}(1/t) \xi_{\rho(t)} \ \underset{t \to \infty}{\overset{\mathrm{(d)}}\longrightarrow} \ D(\gamma).
$$
\item[\emph{(ii)}] \emph{(Strong law of large numbers and fluctuations when $\gamma=1$)}. When $\gamma=1$, this reads  $\Phi^{-1}(1/t) \xi_{\rho(t)} \rightarrow 2$ in probability since $D(1)=\delta_{1/2}$. If we assume additionally that $m:=c+\int_0^{\infty}x \pi(\mathrm dx)<\infty$, then $\Phi^{-1}(1/t)\sim \sqrt{2/mt}$ \ and the convergence holds almost surely:
$$
\frac{ \xi_{\rho(t)}}{\sqrt t} \ {\overset{\mathrm{a.s.}}{\underset{t \to \infty}\longrightarrow}} \ \sqrt{2m}.
$$
Assuming further that $a:=\int_0^{\infty}x^2 \pi(\mathrm dx)<\infty$ and \eqref{hyp:tech}, we have a central limit theorem:
$$
  t^{1/4}\left(\frac{ \xi_{\rho(t)}}{\sqrt t} - \sqrt{2m} \right)  \ \overset{\mathrm{(d)}}{\underset{t \to \infty}\longrightarrow} \ \mathcal N\Bigg(0,\frac{\sqrt 2 a}{3 \sqrt m}\Bigg).
$$
\end{enumerate}
\end{theorem}

Finally, we return to our fragmentation model by considering a fragmentation process $F$ with parameters $(\nu,c, |\log|^{-1})$ and its typical fragment $e^{-\xi^{(\nu,c)}_{\rho}}$. The last theorem can be used to describe the large time asymptotics of the empirical distribution of the whole fragmentation process, following standard methods developed in \cite{BertoinAB03}. In the forthcoming proposition, we let $\phi_{(\nu,c)} (q)=cq+\int_{\mathcal{S}} \sum_{i\geq 1} s_i\big(1-s_i^{q}\big)\nu(\dd \mathbf{s})$ denote the Laplace exponent of $\xi^{(\nu,c)}$ and define $m$ and $a$ from its drift and L\'evy measure as above in the previous theorem. We also let $\Phi_{(\nu,c)}$ be the primitive of $\phi_{(\nu,c)}$ null at 0 and $\Phi^{-1}_{(\nu,c)}$ its inverse.
The limits below hold for the topology of  weak convergence of probability measures.

\begin{proposition}
\label{thm:cvdps}
\begin{enumerate}
\item[\emph{(i)}]
Assume that  $\phi_{(\nu,c)}$ is regularly varying at 0 with index $\gamma \in (0,1]$. Then,
$$
 \sum_{i \geq 1} F_i(t) \delta_{\Phi_{(\nu,c)}^{-1}(1/t) \left |\log F_i(t) \right|}   \ \overset{\mathbb P}{\underset{t \to \infty}\longrightarrow} \ D(\gamma).
$$
\item[\emph{(ii)}] Assuming further that $\gamma = 1$, $a<\infty$ \emph{(}so $m<\infty$\emph{)} and \eqref{hyp:tech} for $\phi_{(\nu,c)}$,
$$
 \sum_{i \geq 1} F_i(t) \delta_{\frac{\left |\log F_i(t) \right| -t^{1/2}\sqrt{2m}}{t^{1/4}}}  \ \overset{\mathbb P}{\underset{t \to \infty}\longrightarrow} \  \mathcal N\Bigg(0,\frac{\sqrt 2 a}{3 \sqrt m}\Bigg).
$$
\end{enumerate}
\end{proposition}

\subsection{Related results on (in)homegeneous fragmentation processes}

To conclude, we compare the large time behavior of the inhomogeneous fragmentation processes we considered here, where particles with mass $m$ split at rate proportional to $|\log m|^{-1}$, to related works on (in)homogeneous fragmentation processes. The most classical class of (in)homogeneous fragmentations are \emph{self-similar fragmentations}, in which particles of mass $m$ split at rate proportional to $m^\alpha$ for some $\alpha \in \R$. When $\alpha=0$,  one recovers the homogeneous fragmentation processes introduced in Section~\ref{sec:intro_frag}. When $\alpha \neq 0$, the processes are constructed from homogeneous ones via the time-change (\ref{time_change_tau}) with $\tau(m)=m^{\alpha}$. These processes naturally satisfy a self-similarity property and as so play an important role as  scaling limits of various models. Their large time behaviors have been well-studied, see notably \cite{BertoinAB03}, and can be summarized as follows when the logarithm of the jumps of a typical fragment has a finite mean:
\begin{enumerate}
\item[$\bullet$] when $\alpha=0$, the masses of particles are asymptotically proportional to $e^{-vt}$ for some $v>0$, where $t$ denotes the time: see \cite[Theorem 1]{BertoinAB03} for the analogue of Proposition~\ref{thm:cvdps} in this situation; see further \cite{Berestycki03, Krell08} for a deep analysis of the different possible exponential rates of decrease
\item[$\bullet$] when $\alpha>0$, the masses of particles are asymptotically proportional to $t^{-1/\alpha}$: see \cite[Theorem 3]{BertoinAB03} for the analogue of Proposition~\ref{thm:cvdps} in this situation
\item[$\bullet$] when $\alpha<0$, the processes gets extincted in finite time: see \cite[Proposition 2]{BertoinAB03}; see further \cite{H03,H23tail} for informations on the extinction times and \cite{GH1,GH2} for the behavior of the processes in the neighborhood of their extinction time.
\end{enumerate}
Processes with fragmentation rates proportional to $|\log m|^{-1}$, where our results roughly says that the masses are asymptotically proportional to $e^{-v\sqrt t}$ for some $v>0$ (when the logarithm of the jumps of a typical fragment has a finite mean) can therefore be interpreted as an interpolation between homogeneous fragmentations and self-similar fragmentations with a positive index $\alpha$. Note however that our Proposition~\ref{thm:cvdps} (i) is relatively simple to implement and that the main results of this paper concern the exact computation, at any time $t$, of the positive powers of the typical fragment (via Theorem~\ref{thm:lapTransform}) and their asymptotic behaviors as $t\rightarrow \infty$ (Theorem~\ref{thm:asympBehavior}). Similar results can trivially be settled for homogeneous fragmentations but seem less obvious for self-similar ones with indices $\alpha \neq 0$.

In a related way, we observe that "dual" fragmentation processes in which particles  with mass $m$ split at rate proportional to $|\log m|$ also appeared in the literature under another guise. Indeed, the time-change (\ref{time_change_tau}) with $\tau(m)=|\log m|$ leads to a typical fragment $e^{-\xi_{\eta}}$ where $\xi$ is still a subordinator and $ \eta(t) = \inf \left\{ r \geq 0 : \int_0^r (\xi_u)^{-1} \dd u \geq t  \right\}$. We recognize here the celebrated Lamperti's transform relating L\'evy processes with no negative jumps to continuous-state branching processes, see e.g. \cite{Lamp67bis,CLUB09}. In other words, $\xi_{\eta}$ is a CSBP and one can compute explicitly the Laplace transforms of $\xi_{\eta(t)}$ at any time $t$ when $\xi_0=x>0$. Indeed, it is well-known that then $\mathbb E\big[e^{-\lambda \xi_{\eta(t)}}\big]=e^{-xu_t(\lambda)}$, where the function $t \mapsto u_t(\lambda)$ is characterized by the equation
$$
\frac{\partial u_t(\lambda)}{\partial t}=\phi( u_t(\lambda)), \quad u_0(\lambda)=\lambda,
$$
with $\phi$ the Laplace exponent of $\xi$.

\bigskip

\paragraph{Organization of the paper.} We prove in Section~\ref{sec:Laplace}~ Theorem~\ref{thm:lapTransform} and Proposition~\ref{prop:unifBound}. This is then exploited in Section~\ref{sec:asymptotic_sub} to obtain Theorem~\ref{thm:asympBehavior} and Theorem~\ref{thm:asymp_distribution}. This last result will in turn allow us to describe the asymptotic behavior of the whole fragmentation process in Section~\ref{sec:wholefrag} by proving Proposition~\ref{thm:cvdps}. Last, Section~\ref{sec:discussions} gathers some final discussions on the model and possible extensions to higher dimensional settings.

\section{The time-changed subordinator \texorpdfstring{$\xi_{\rho}$}{xiRho} and its Laplace exponent}
\label{sec:Laplace}

In this section and the next section, we work with $\xi$ a generic subordinator with Laplace exponent
\[
  \phi(q)=\kappa+cq+\int_0^{\infty} \left(1-\exp(-qx)\right)\pi(\mathrm dx), \qquad q \geq 0
\]
where $\kappa \geq 0$, $c \geq 0$ and $\pi$ is a measure on $(0,\infty)$ such that $\int_0^{\infty} (1 \wedge x) \pi(\mathrm dx)<\infty$. We assume throughout that $\pi(0,\infty)>0$ and $\xi_0=0$. We recall that $\rho(t)$ is the stopping time defined by
\begin{equation}
\label{eq:time_change}
\rho(t)=\inf \left\{u \geq 0 : \int_0^u \xi_r \mathrm dr >t \right\} \text{ for all }t\geq 0.
\end{equation}
Observe that $\rho(0) = \inf\{u \geq 0 : \xi_u > 0\}$ is either null almost surely, if $c > 0$ or $\pi((0,\infty)) = \infty$, or is strictly positive almost surely, when $c=0$ and $\pi$ is finite, equal to the first jump time of $\xi$, in which case $\xi_{\rho(0)}$ is equal to the first jump of $\xi$. Observe also that
\[
 \int_0^{\rho(t)} \xi_r \mathrm dr =t  \quad \text{for all} \quad t <\int_0^{\mathbf e(\kappa)}\xi_r \mathrm dr,
\]
where $\mathbf e(\kappa):=\inf\{t \geq 0:\xi_t=\infty\}$ denotes the death time of $\xi$. In particular when $\kappa=0$, $\mathbf e(\kappa)=\infty$ and $\xi_{\rho}$ increases asymptotically much slower than $\xi$, since $\xi_t \to \infty$ and $\rho(t)/t \to 0$ as $t \to \infty$, almost surely.

The main objective of this section is to compute the Laplace transform of $\xi_{\rho(t)}$ for any $t\geq 0$ by proving Theorem~\ref{thm:lapTransform}, and then to prove the upper bound settled in Proposition~\ref{prop:unifBound}. We will first start by introducing the measure $L$ in Section~\ref{sec:L}. It is worth noting that an explicit formula for this measure is generally inaccessible. However, before turning to the proof of Theorem~\ref{thm:lapTransform}, we consider in Section~\ref{sec:stable} a situation in which all quantities can be made explicit: this is the case when $\xi$ is a stable subordinator. In this case we also note that the process $(t^{-1/(1+\gamma)}\xi_{\rho(t)},t >0)$ is stationary, where $\gamma$ is the index of stability of $\xi$. The proof of  Theorem~\ref{thm:lapTransform} is then undertaken in Section~\ref{sec:proofThmain} and that of Proposition~\ref{prop:unifBound} in Section~\ref{sec:bounds}.

\subsection{Definition and first properties of the measure \texorpdfstring{$L$}{L}}
\label{sec:L}

We recall from the statement of Theorem~\ref{thm:lapTransform} that $\Phi$ is the primitive of $\phi$ satisfying $\Phi(0) = 0$, i.e.
\begin{align}
\label{LK_Phi}
\Phi(q)&=\kappa q+\frac{c}{2}q^2+\int_{\mathbb R_+} \left(q-\frac{\left(1-\exp(-qx)\right)}{x}\right)\pi(\mathrm dx) \\
\nonumber
&= \kappa q+\frac{c}{2}q^2+ q \pi([1,\infty))+\int_0^{\infty} \left(qx \mathbbm 1_{\{x<1\}} -\left(1-\exp(-qx)\right)\right) \frac{\pi(\mathrm dx)}{x}.
\end{align}
A key remark is that $\Phi$ is the Laplace transform of a Lévy process with diffusion coefficient $c$, drift $\kappa + \pi([1,\infty))$ and a L\'evy measure given by the image of $\pi(\dd x)/x$ by $x \mapsto -x$. We denote $X$ such a process, satisfying
\begin{equation*}
\label{def_X}
\mathbb E\big[e^{qX_t}\big]=e^{t \Phi(q)} \quad \text{for all } q,t \geq 0.
\end{equation*}
This class of processes, usually called \emph{spectrally negative L\'evy processes}, has remarkable properties and has been studied extensively. We refer to Bertoin's book \cite[Chapter VII]{BLevy} for background. Here, we note that either $X$ oscillates ($\limsup_{t\rightarrow \infty} X_t=-\liminf_{t\rightarrow \infty} X_t=+\infty$) when $\kappa=0$, or tends to $+\infty$ ($\lim_{t\rightarrow \infty} X_t=+\infty$) when $\kappa>0$, using that $\Phi'(0)=\phi(0)=\kappa$. It is well known that $\Phi^{-1}:[0,\infty)\mapsto [0,\infty)$ is the Laplace exponent of a subordinator. More precisely,

\begin{lemma}[Theorem 1, Chapter VII of \cite{BLevy}]
The process $\sigma$ defined for $t \geq 0$ by
\[
   \sigma_t = \inf\{ s > 0 : X_s > t\}=\inf\{ u >0:S_u>t\},
\]
where $S_t = \sup_{s \leq t} X_s$, is a subordinator with Laplace exponent $\Phi^{-1}$. Moreover $S$ is a local time at 0 of the reflected process $S-X$.
Therefore, if we let $L$ denote the L\'evy measure of $\sigma$,  we have
\[
  L(\dd x)=n(\zeta \in \mathrm dx)
\]
where $n$ denotes the excursion measure away from 0 of $S-X$ and $\zeta$ the lifetime of an excursion.
\end{lemma}

Since $\Phi^{-1}(0)=0$ and, as $q\rightarrow \infty$, $\Phi^{-1}(q)/q$ converges to $(\kappa+\pi(0,\infty))^{-1}$ when $c=0$ and to $0$ otherwise, the subordinator $\sigma$ has no killing term and a drift equal to $\mathbbm 1_{\{c=0\}}(\kappa+\pi(0,\infty))^{-1}$. Consequently the function $\Phi^{-1}$ can be written for all $q\geq 0$ as
\begin{equation}
\label{LK_invPhi}
\Phi^{-1}(q)=\frac{q \mathbbm 1_{\{c=0\}}}{\kappa+\pi(0,\infty)}+\int_0^{\infty} \left(1-\exp(-qx)\right)L(\mathrm dx).
\end{equation}
The measure $L$ plays an important role in our study. We underline below a few simple observations. A first one, which will be useful for the proof of Proposition~\ref{prop:unifBound}, is the following connexion with a renewal measure.

\newpage

\begin{lemma}
\label{lem:renew}
Let $U$ be the  renewal measure of the subordinator $\sigma \circ \xi$ when we consider independent copies of $\xi$ and $\sigma$, and set $V(\mathrm dx):=xL(\mathrm dx)$. Then
$$
U=V + \frac{\mathbbm 1_{\{c=0\}}}{\kappa+\pi(0,\infty)} \delta_0.$$
Consequently, the function $\bar V$, defined by $\bar V(x):=V([0,x])$ for $x \geq 0$, satisfies:
$$
\bar V(x+y)\leq \bar V(x)+ \bar V(y) + \frac{\mathbbm 1_{\{c=0\}}}{\kappa+\pi(0,\infty)} \qquad \text{for all } x,y \geq 0.
$$
\end{lemma}

\begin{proof}
For background on renewal measures of subordinators we refer to \cite[Chapter 3]{BLevy} or \cite[Chapter 5]{KypFluctu}.
By differentiating (\ref{LK_invPhi}), we note that
$$
\frac{1}{\phi \circ \Phi^{-1}(q)}=(\Phi^{-1})'(q)=\frac{\mathbbm 1_{\{c=0\}}}{\kappa+\pi(0,\infty)}+ \int_0^{\infty} e^{-qx} V(\mathrm dx).
$$
The function $\phi \circ \Phi^{-1}$ being the Laplace transform of the subordinator $\sigma \circ \xi$ when we consider independent copies of $\xi$ and $\sigma$, this implies that the measure $V$ is the contribution on $(0,\infty)$ of the renewal measure $U$ of $\sigma \circ \xi$ and that $U(\{0\})=\frac{\mathbbm 1_{\{c=0\}}}{\kappa+\pi(0,\infty)}$. It is well-known that the distribution function $x \mapsto U([0,x])$ is then subadditive. This has the particular consequence that the function $\bar V$ satisfies
\[\bar V(x+y)+ \frac{\mathbbm 1_{\{c=0\}}}{\kappa+\pi(0,\infty)} \leq \bar V(x)+ \frac{\mathbbm 1_{\{c=0\}}}{\kappa+\pi(0,\infty)}+ \bar V(y) + \frac{\mathbbm 1_{\{c=0\}}}{\kappa+\pi(0,\infty)} \quad \text{for all } x,y \geq 0. \qedhere  \]
\end{proof}

We finish with easy integrability properties of $L$. Using that $\Phi^{-1}$ is the Laplace exponent of the subordinator $\sigma$ of  first passage times of the Lévy process $X$, we remark that the process $t \mapsto \sigma_t - \frac{t \ind{c = 0}}{\kappa + \pi(0,\infty)}$ is a compound Poisson process if and only if either the L\'evy measure of $X$ is infinite or its diffusion coefficient $c$ is strictly positive. In other words, we have
\[
  L(0,\infty) < \infty \quad \text{ if and only if } \quad c = 0 \quad \text{ and } \quad \int_0^\infty x^{-1} \pi(\dd x) <\infty.
\]
Moreover, in this case we have
\[
  L(0,\infty) = \lim_{q \to \infty} \left(\Phi^{-1}(q) - \frac{q}{\kappa + \pi(0,\infty)}\right) = \frac{1}{\kappa + \pi(0,\infty)} \int_0^\infty x^{-1} \pi(\dd x).
\]
Similarly, we have $\int_0^{\infty} x L(\dd x) < \infty$ if and only if $\E[\sigma_t] < \infty$ for all $t$. Differentiating (\ref{LK_invPhi}), we see that
\[
  \int_0^{\infty} x L(\dd x) = (\Phi^{-1})'(0) - \frac{1}{\kappa + \pi(0,\infty)} = \frac{1}{\kappa} - \frac{\ind{c = 0}}{\kappa + \pi(0,\infty)}.
\]
In particular, $\int_0^{\infty} x L(\dd x) < \infty$ if and only if $\kappa > 0$.

\subsection{An explicit example: the stable subordinator}
\label{sec:stable}

In this paragraph we assume that $\xi$ is a stable subordinator with parameter $\gamma \in (0,1]$. Since the multiplication of the Laplace exponent $\phi$ by a positive constant results in a multiplication of time by the same constant in the time-changed subordinator,  we restrict ourselves, with no loss of generality, to the canonical case
$$
\phi(q)=(\gamma+1)q^{\gamma}, \qquad q \geq 0.
$$
We let $\xi^{(\gamma)}$ denote a subordinator with this Laplace exponent. In this case, one has
\[
  \Phi(q) = q^{\gamma  +1} \quad \text{ and } \quad \Phi^{-1}(q) = q^\frac{1}{\gamma + 1} = \frac{1}{(\gamma+1)\Gamma(\frac{\gamma}{\gamma+1})}\int_0^\infty (1 - e^{-q x}) \frac{\dd x }{x^{\frac{1}{\gamma  +1}+1}}.
\]
In particular, the Lévy measure of $\Phi^{-1}$ is $L(\dd x) = \frac{1}{(\gamma+1)\Gamma(\frac{\gamma}{\gamma+1})} \cdot \frac{\dd x }{x^{\frac{1}{\gamma  +1}+1}}$. Therefore, applying Theorem~\ref{thm:lapTransform}, the following holds.

\begin{corollary}
\label{cor:stable}
The process $\big(t^{-\frac{1}{\gamma+1}}\xi^{(\gamma)}_{\rho(t)}, t >0\big)$ is stationary, and its stationary distribution $D(\gamma)$ is characterized by its Laplace transform
\[
q \mapsto \frac{q^{\gamma}}{\Gamma\left(\frac{\gamma}{\gamma+1} \right)} \int_0^{\infty} e^{-q^{\gamma+1}u-\frac{1}{u}} \frac{ \mathrm du}{u^{\frac{1}{1+\gamma}}} =  \frac{2}{\Gamma\left(\frac{\gamma}{\gamma+1} \right)} q^{\frac{\gamma}{2}}K_{\frac{\gamma}{\gamma+1}}\left(2q^{\frac{\gamma+1}{2}}\right)
\]
where, for all $\alpha \geq 0$, $K_{\alpha}$ denotes the modified Bessel function of the second kind.
\end{corollary}

We recall from \cite[Chapter 9.6]{handbook} that the modified Bessel functions of the second kind of order $\alpha$ is the function defined for $x \in (0,\infty)$ by
\begin{equation*}
\label{K_int}
K_{\alpha}(x) = \int_0^{\infty} e^{-x \cosh(u)}\cosh(\alpha u) \mathrm du, \qquad x>0,
\end{equation*}
and that its asymptotic behaviors near $\infty$ and $0$ are given by
$$K_{\alpha}(x) \underset{x \to \infty}{\sim} \sqrt \frac{\pi }{2x}e^{-x}\left(1+\frac{4 \alpha^2-1}{8x}+O(x^{-2}) \right) \quad \text{and} \quad K_{\alpha}(x) \underset{x \to 0}{\sim} \frac{\Gamma(\alpha)}{2} \left(\frac{2}{x} \right)^{\alpha}.$$

\medskip

\begin{remark}
The modified Bessel function of the second kind of order $\alpha=1/2$ simply rewrites
$K_{1/2}(x) = \sqrt \frac{\pi }{2x}e^{-x},x>0$. Note that this implies that in Corollary~\ref{cor:stable}
$$D(1)=\delta_{2}.$$
Of course, this is consistent with the definition of the process $\xi^{(1)}_{\rho}$, since $\xi^{(1)}_t = 2t$ and then $\rho(t)=\sqrt t$, $t\geq 0$.
\end{remark}

\begin{proof}[\bf Proof of Corollary~\ref{cor:stable}]
The main point of this result concerns the identification, via its Laplace transform, of the stationary law $D(\gamma)$. Indeed, even if the stationarity of $(\xi^{(\gamma)}_{\rho(t)}{t^{-\frac{1}{\gamma+1}}})$ is an immediate consequence of Theorem~\ref{thm:lapTransform}, it can also be seen directly from the definition (\ref{eq:time_change}) of the time change $\rho$ and the self-similarity of $\xi^{(\gamma)}$, since for all $a>0$, $(a^{-\frac{1}{\gamma}} \xi^{(\gamma)}_{at},t\geq 0)$ is distributed as $\xi^{(\gamma)}$. Using this identity in law of processes and the definition of the time change $\rho$, we note that $\xi^{(\gamma)}_{\rho}$ itself is a self-similar Markov process: for all $a>0$,
$$
\Big(a^{-\frac{1}{1+\gamma}} \xi^{(\gamma)}_{\rho(at)},t\geq 0 \Big) \overset{\mathrm{(d)}} = \xi^{(\gamma)}_{\rho}.
$$
We now apply Theorem~\ref{thm:lapTransform} to compute the distribution $D(\gamma)$ of $\xi^{(\gamma)}_{\rho(1)}$: for $q>0$,
\begin{align*}
  \E\left[ e^{-q \xi^{(\gamma)}_{\rho(1)}} \right] &=  \phi(q)\int_0^{\infty} e^{-\Phi(q)x-\frac{1}{x}} xL(\mathrm dx)\\
  &= \frac{q^\gamma }{\Gamma(\frac{\gamma}{\gamma+1})} \int_0^\infty e^{-q^{\gamma + 1} x-\frac{1}{x}}  \frac{\dd x}{x^{\frac{1}{\gamma + 1}}}.
\end{align*}
Therefore, using that for all $\alpha > 0$ and $r,s > 0$, one has
\begin{equation*}
\label{Bessel2}
\int_0^{\infty} x^{\alpha-1} e^{-rx-\frac{s}{x}} \mathrm dx = 2 \left( \frac{s}{r}\right)^{\frac{\alpha}{2}} K_{\alpha} \left(2 \sqrt{rs}\right),
\end{equation*}
see e.g. \cite{Temme90}, we obtain the Laplace transform of the distribution $D(\gamma)$.
\end{proof}

\subsection{Proof of Theorem~\ref{thm:lapTransform}}
\label{sec:proofThmain}

We now prove Theorem~\ref{thm:lapTransform}. In that aim we first introduce, for all $t \geq 0, q>0$,
\begin{equation}
  \label{def:f}
  f(t,q):=(\Phi^{-1})'(q)\mathbb E\left[e^{-\Phi^{-1}(q) \xi_{\rho(t)}}\right]
\end{equation}
and observe that this bivariate function $f$ is the solution to a simple PDE.

\begin{lemma}
\label{lem:easy}
The function $f$ is a strong solution to the equation
\begin{equation}
\label{eq:easy}
 \partial_q \partial_t  f = \partial_t \partial_q f=  f.
\end{equation}
Additionally, the function $t \in [0,\infty) \mapsto f(t,q)$ is $\mathcal C^1$ on $(0,\infty)$ and continuous at 0, with $f(0,q)=\int_0^{\infty}xe^{-qx}L(\mathrm dx)$,  for all $q>0$.
\end{lemma}

In the proof of this lemma, we will note some regularity properties of $f$ that are a bit more precise than those stated here. Moreover, we will see further with Lemma \ref{lem:laplaceLaplace}  that $f$ is infinitely differentiable on $(0,\infty)^2$.

\medskip

\begin{remark}
The form \eqref{eq:easy} of the partial differential equation solved by $f$ comes from the time change $|\log |^{-1}$ in the fragmentation process. If we have used a time change of the form $|\log |^{-k}$ for some $k \in \N$,
we would have obtained for the Laplace transform $h(t,q):=\mathbb E\big[e^{-q \xi_{\rho(t)}}\big]$ the equation
\[
 \partial_q^k  \left( \frac{\partial_t h}{\phi(q)}\right)= (-1)^{k+1}h,
\]
with similar arguments.
\end{remark}

\smallskip

\begin{remark}
\label{rem:gene}
Observe that the infinitesimal generator $\mathcal G^{\exp(-\xi_{\rho})}$ of $\exp(-\xi_{\rho})$ is given by
\begin{equation*}
\label{generator}
\mathcal G^{\exp(-\xi_{\rho})}(g)(x)=\frac{1}{\left|\log(x)\right|}\left(-\kappa g(x)-cxg'(x)+\int_0^{\infty} \left(g(x\exp(-y))-g(x)\right) \pi (\mathrm dy)\right)
\end{equation*}
at least for functions $g:[0,1] \rightarrow \mathbb R$ that are continuously differentiable, null in a neighborhood of 1 and such that $g(0)=g'(0)=0$, using the formula for the generator of a Lévy process and classical results on time-changed processes from Lamperti \cite{Lamp67}. Lemma~\ref{lem:easy} could therefore be proved by applying Kolmogorov's forward equation to $x \mapsto x^q$, $q>0$ provided that we can show that these functions belong to the domain of the generator. This is doable by approximation but requires the use of fine estimates to apply the dominated convergence theorem.
\end{remark}

Instead of using this approach based on the infinitesimal generator, we prefer to use a more elementary method, which consists in rewriting  the function $f$ as follows.
\begin{lemma}
\label{lem:rewritten}
For all $t \geq 0$ and $q > 0$, the function $f$ defined in \eqref{def:f} satisfies
\[
  f(t,q) = \int_t^\infty \E\left[ \frac{e^{-\Phi^{-1}(q)\xi_{\rho(s)}}}{\xi_{\rho(s)}} \right] \dd s.
\]
\end{lemma}

\begin{proof} Fix $q>0, t \geq 0$.
We first observe that the change of variable $\rho(s) = u$ leads to
\[
  \int_t^\infty \frac{e^{-\Phi^{-1}(q) \xi_{\rho(s)}}}{\xi_{\rho(s)}} \dd s = \int_{\rho(t)}^{\infty} e^{-\Phi^{-1}(q) \xi_u} \dd u.
\]
We underline that when $\kappa > 0$ the subordinator $\xi$ reaches $\infty$ in finite time and that the identity above is valid by using the convention $e^{-\infty} = 0$.
Then by Fubini's theorem
\[
  \int_t^\infty \E\left[ \frac{e^{-\Phi^{-1}(q)\xi_{\rho(s)}}}{\xi_{\rho(s)}} \right] \dd s =
  \E\left[\int_{\rho(t)}^{\infty} e^{-\Phi^{-1}(q) \xi_u} \dd u \right].
\]
Applying the strong Markov property for Lévy processes to the stopping time $\rho(t)$, we obtain
\begin{align*}
  \E\left[\int_{\rho(t)}^{\infty} e^{-\Phi^{-1}(q) \xi_u} \dd u \right]
  &= \E\left[ e^{-\Phi^{-1}(q) \xi_{\rho(t)}} \int_0^{\infty} e^{-\Phi^{-1}(q) (\xi_{u +\rho(t)} - \xi_{\rho(t)})} \dd u \right]\\
  &= \E\left[ e^{-\Phi^{-1}(q) \xi_{\rho(t)}} \right] \int_0^\infty \E\left[ e^{-\Phi^{-1}(q) \xi_{u}} \right] \dd u.
\end{align*}
Observing that $\E\left[e^{- \Phi^{-1}(q) \xi_t} \right] = e^{-t \phi(\Phi^{-1}(q))}=e^{-t \Phi'(\Phi^{-1}(q))}$, we conclude that
\[
  \int_t^\infty \E\left[ \frac{e^{-\Phi^{-1}(q)\xi_{\rho(s)}}}{\xi_{\rho(s)}} \right] \dd s = \frac{1}{\Phi' \circ \Phi^{-1}(q)}\E\left[ e^{-\Phi^{-1}(q) \xi_{\rho(t)}} \right] = f(t,q). \qedhere
\]
\end{proof}

\medskip

\begin{proof}[\bf Proof of Lemma~\ref{lem:easy}]
Using that for all $a,b > 0$, $x \mapsto x^{a} e^{-bx}$ is bounded on $\R_+$, we obtain easily that the function $q \in (0,\infty) \mapsto f(t,q)$ is infinitely differentiable. Moreover, by Lemma~\ref{lem:rewritten}, for all $t \geq 0, q>0$,
\[
  \int_t^\infty \E\left[ \frac{e^{-\Phi^{-1}(q)\xi_{\rho(s)}}}{\xi_{\rho(s)}} \right] \dd s = f(t,q) < \infty,
\]
therefore $\E[e^{-\Phi^{-1}(q)\xi_{\rho(t)}}/\xi_{\rho(t)}] < \infty$ for almost all $t \geq 0$. Remarking that $t \mapsto e^{-\Phi^{-1}(q)\xi_{\rho(t)}}/\xi_{\rho(t)}$ is non-increasing, $t  \mapsto \E[e^{-\Phi^{-1}(q)\xi_{\rho(t)}}/{\xi_{\rho(t)}}]$ is a non-increasing function. As a consequence, $t \in (0,\infty) \mapsto f(t,q)$ is a convex decreasing function and differentiable almost everywhere with
\[
  \partial_t f(t,q) = -\E\left[ \frac{e^{-\Phi^{-1}(q)\xi_{\rho(t)}}}{\xi_{\rho(t)}} \right]  \quad \text{for almost all $t > 0$}.
\]
We then observe that for all $t,q > 0$, using Fubini's theorem
\begin{align*}
  \int_q^\infty f(t,r) \dd r &= \int_q^\infty (\Phi^{-1})'(r)\E\left[e^{-\Phi^{-1}(r) \xi_{\rho(t)}}\right] \dd r \\
  &= \E\left[ \int_q^\infty (\Phi^{-1})'(r)e^{-\Phi^{-1}(r) \xi_{\rho(t)}} \dd r \right] = \E\left[ \frac{e^{-\Phi^{-1}(q)\xi_{\rho(t)}}}{\xi_{\rho(t)}} \right].
\end{align*}
The function $t \in (0,\infty) \mapsto \E\left[ {e^{-\Phi^{-1}(q)\xi_{\rho(t)}}}/{\xi_{\rho(t)}} \right]$ is therefore continuous, and consequently the function $t \in (0,\infty) \mapsto f(t,q)$ is $\mathcal{C}^1$ for all $q > 0$ and
\[
  \partial_t f(t,q) = -\int_q^\infty f(t,r) \dd r.
\]
Note that this implies that $q \in (0,\infty) \mapsto \partial_t f(t,q)$ is infinitely differentiable. In particular, differentiating the above expression with respect to $q$ shows that $\partial_q \partial_t f(t,q) = f(t,q)$.

Similarly, by Fubini's theorem and the change of variables $\rho(s)=u$,
\begin{eqnarray*}
  \int_t^\infty f(s,q) \dd s &=& (\Phi^{-1})'(q)\mathbb E\left[ \int_t^\infty e^{-\Phi^{-1}(q) \xi_{\rho(s)}} \dd s\right]\\
  &=& (\Phi^{-1})'(q) \E\left[ \int_{\rho(t)}^{\infty} \xi_u e^{-\Phi^{-1}(q) \xi_u}  \dd u \right]
\end{eqnarray*}
for $t,q>0$.
Moreover, we note that for $u> 0$
\begin{eqnarray*}
  \E\left[\xi_{\rho(t)+u} e^{-\Phi^{-1}(q) \xi_{\rho(t)+u}}\right]&=& \E\left[ \left(\xi_{\rho(t)} +(\xi_{\rho(t) + u} - \xi_{\rho(t)}) \right] e^{-\Phi^{-1}(q) \xi_{\rho(t)}} e^{-\Phi^{-1}(q) (\xi_{\rho(t) + u} - \xi_{\rho(t)})} \right]\\
  &=& \E\left[ \xi_{\rho(t)} e^{-\Phi^{-1}(q)\xi_{\rho(t)}} \right] e^{-u \phi(\Phi^{-1}(q))} \\
  & +& \E\left[ e^{-\Phi^{-1}(q) \xi_{\rho(t)}}\right] u \phi'(\Phi^{-1}(q)) e^{-u \phi(\Phi^{-1}(q))}.
\end{eqnarray*}
This leads to
\begin{align*}
    \int_t^\infty f(s,q) \dd s  &= \frac{(\Phi^{-1})'(q)}{\phi(\Phi^{-1}(q))} \E\left[ \xi_{\rho(t)} e^{-\Phi^{-1}(q)\xi_{\rho(t)}} \right] + \phi'(\Phi^{-1}(q)) \frac{(\Phi^{-1})'(q)}{\phi(\Phi^{-1}(q))^2} \E\left[ e^{-\Phi^{-1}(q) \xi_{\rho(t)}} \right]\\
    &= -\partial_q f(t,q),
\end{align*}
since by definition of $f$
\[
  \partial_q f(t,q) = (\Phi^{-1})''(q)\E\left[ e^{-\Phi^{-1}(q) \xi_{\rho(t)}}\right] - (\Phi^{-1})'(q)^2 \E\left[ \xi_{\rho(t)} e^{-\Phi^{-1}(q)\xi_{\rho(t)}} \right]
\]
and $(\Phi^{-1})' = 1/\phi (\Phi^{-1})$. Differentiating with respect to $t$ shows that $\partial_t \partial_q f(t,q) = f(t,q)$ as well.

The continuity at 0 of $t \mapsto f(t,q)$ is obvious, for all $q>0$, by using the expression of $f$ as an integral as noticed in Lemma~\ref{lem:rewritten}. It remains to identify the initial expression for $f(0,q)$ in terms of the measure $L$. Since $\xi_{\rho(0)}=0$ when $c>0$ or $\pi$ is infinite, and $\xi_{\rho(0)}$ is the first jump time of $\xi$ otherwise, one has
$$
\mathbb E[e^{-q \xi_{\rho(0)}}]=\left\{\begin{array}{ll} 1 & \text{when $c>0$ or  $\pi$ is infinite}\\ \frac{1}{\kappa+\pi(0,\infty)}\int_0^{\infty}e^{-qx} \pi(\mathrm dx) & \text{when $c=0$ and $\pi$ is finite.}\end{array}\right.
$$
When $c>0$ or $\pi$ is infinite, recalling \eqref{LK_invPhi}, this indeed yields
$$
f(0,q)=(\Phi^{-1})'(q)=\int_0^{\infty} xe^{-qx} L(\mathrm dx).
$$
When $c=0$ and $\pi$ is finite, we have
$$
f(0,q)=(\Phi^{-1})'(q) \frac{1}{\kappa+\pi(0,\infty)}\int_0^{\infty}e^{-\Phi^{-1}(q)x} \pi(\mathrm dx).
$$
Using that in this case $\Phi(q)=(\kappa+\pi(0,\infty))q-\int_0^{\infty}(1-e^{-qx})\pi(\mathrm dx)/x$, we see that
$$
q=(\kappa+\pi(0,\infty))\Phi^{-1}(q)-\int_0^{\infty}(1-e^{-\Phi^{-1}(q)x})\pi(\mathrm dx)/x,
$$
which, by differentiating, leads us to
$$
f(0,q)=(\Phi^{-1})'(q) \frac{1}{\kappa+\pi(0,\infty)}\int_0^{\infty}e^{-\Phi^{-1}(q)x} \pi(\mathrm dx)=(\Phi^{-1})'(q) -\frac{1}{\kappa+\pi(0,\infty)}.
$$
By \eqref{LK_invPhi} and since $c=0$, this last expression is equal to $\int_0^{\infty} xe^{-qx} L(\mathrm dx)$ as expected.
\end{proof}

Solving the partial differential equation of Lemma~\ref{lem:easy}, we obtain an expression for the Laplace transform of the function $t \mapsto f(t,q)$ for each $q > 0$. This leads to an explicit expression of $f(t,q)$ in terms of the measure $L$, which proves Theorem~\ref{thm:lapTransform}.

\begin{lemma}
\label{lem:laplaceLaplace}
For $\lambda>0$, we write $F_{\lambda} : q \in (0,\infty)\mapsto \int_0^{\infty} e^{-\lambda t} f(t,q) \mathrm dt$. Then for all fixed $q>0$, we have
\begin{equation}
\label{eq:Laplace}
F_{\lambda}(q)=\int_0^{\infty} \left( \int_0^{\infty} e^{-qx -\frac{t}{x}} xL(\mathrm dx)\right) e^{-\lambda t} \mathrm dt, \quad \forall \lambda>0
\end{equation}
which implies that $f(t,q)= \int_0^{\infty} e^{-qx -\frac{t}{x}} xL(\mathrm dx)$ for all $t \geq 0, q>0$. In particular, $f$ is infinitely differentiable on $(0,\infty)^2$.
\end{lemma}

\begin{proof}
Fix $\lambda > 0$ and note that $F_{\lambda}(q)$ is finite for all $q > 0$ since $t \in (0,\infty) \mapsto f(t,q)$ is bounded.
An integration by part gives
\begin{align*}
F_{\lambda}(q)=&  \left[-\frac{e^{-\lambda t}}{\lambda}f(t,q) \right]_{t=0}^{t=\infty} + \frac{1}{\lambda} \int_0^{\infty}e^{-\lambda t} \partial _t f(t,q) \mathrm dt \\
=& \frac{1}{\lambda} \int_0^{\infty} xe^{-qx} L(\mathrm dx)+\frac{1}{\lambda} \int_0^{\infty}e^{-\lambda t} \partial _t f(t,q) \mathrm dt,
\end{align*}
using that $e^{-\lambda t}f(t,q) \rightarrow 0 $ as $t\rightarrow \infty$, and that $f(t,q) \rightarrow  \int_0^{\infty} xe^{-qx} L(\mathrm dx)$ as $t \rightarrow 0$ by Lemma~\ref{lem:easy}.
We then have, for all $q>0$,
\begin{eqnarray*}
\partial_q F_{\lambda}(q)&=& \frac{-1}{\lambda} \int_0^{\infty} x^2e^{-qx} L(\mathrm dx)+ \frac{1}{\lambda} \int_0^{\infty}e^{-\lambda t} \partial _q\partial _t f(t,q) \mathrm dt \\
&=& \frac{-1}{\lambda} \int_0^{\infty} x^2e^{-qx} L(\mathrm dx)+ \frac{1}{\lambda} F_{\lambda}(q),
\end{eqnarray*}
using again Lemma~\ref{lem:easy} (we can apply Lebesgue's dominated convergence theorem since, for all $q_0>0$ and $t>0$, $\sup_{q \geq q_0} |e^{-\lambda t}\partial_q \partial_t  f(t,q)|=\sup_{q \geq q_0} |e^{-\lambda t} f(t,q)|=e^{-\lambda t} f(t,q_0)$, since $(\Phi^{-1})'$ is decreasing and $\Phi^{-1}$ is increasing).

The function $q \in (0,\infty) \mapsto F_{\lambda}(q)$ is therefore a solution to a linear differential equation of the first order. By standard arguments, it follows that
$$
F_{\lambda}(q)=e^{\frac{q}{\lambda}}\int_q^{\infty} e^{-\frac{u}{\lambda}}  \left(\frac{1}{\lambda} \int_0^{\infty} x^2e^{-ux} L(\mathrm dx) \right) \mathrm du.
$$
Using Fubini's theorem and the change of variable $t=(u-q)x/\lambda$, we then get the expression \eqref{eq:Laplace} for all $q,\lambda>0$. The final equation follows from Laplace inversion theorem and the fact that the functions $t\mapsto f(t,q)$ and $t \mapsto \int_0^{\infty} e^{-qx -\frac{t}{x}} xL(\mathrm dx)$ are continuous on $(0,\infty)$.
\end{proof}

\begin{proof}[{\bf Proof of Theorem~\ref{thm:lapTransform}}]
Observe that the definition \eqref{def:f} of $f$ implies that
\[
  f(\Phi(q),t) = \big(\Phi^{-1}\big)'(\Phi(q)) \E\left[ e^{-q \xi_{\rho(t)}} \right] = \frac{1}{\phi(q)}\E\left[ e^{-q \xi_{\rho(t)}} \right].
\]
As a result, Lemma~\ref{lem:laplaceLaplace} yields
$$
\mathbb E\left[e^{-q\xi_{\rho(t)}}\right] = \phi(q) f\left(t,\Phi(q)\right)=\phi(q)  \int_0^{\infty} e^{-\Phi(q)x -\frac{t}{x}} \ xL(\mathrm dx), \quad \forall t\geq 0, q>0,
$$
which completes the proof.
\end{proof}

\subsection{Proof of Proposition~\ref{prop:unifBound}}
\label{sec:bounds}

The proof of Proposition~\ref{prop:unifBound} relies on Theorem~\ref{thm:lapTransform} and Lemma~\ref{lem:renew}. In fact, we use the forthcoming Lemma~\ref{eq:dev}, an immediate corollary of Theorem~\ref{thm:lapTransform}, which rewrites in a more convenient way the expression of the Laplace transform of $\xi_{\rho(t)}$. From this lemma and Lemma~\ref{lem:renew}, we note that for all $t \geq 0,q>0$
\begin{equation}
\label{eq:largebound}
\mathbb E\left[ e^{-q\xi_{\rho(t)}}\right]
\leq  \phi(q) e^{-2 \sqrt{\Phi(q)t}}\int_0^\infty 2u e^{-u^2} \left(\bar V\left(2 \left(\frac{\sqrt{u^4+4u^2 \sqrt{\Phi(q) t}}}{2 \Phi(q)} \right)\right) +   \frac{\mathbbm 1_{\{c=0\}}}{\kappa+\pi(0,\infty)} \right)\mathrm du,
\end{equation}
where we recall that $\bar V(x)=\int_{[0,x]}u L(\mathrm du)$.
It remains to bound $\bar V$ from above to conclude. Since $\phi$ is concave, the function $q \mapsto \phi(q)/q$ is decreasing on $(0,\infty)$ and $\lim_{q \downarrow 0} \phi(q)/q \in (0,\infty]$. Consequently there exists $c_1 \in (0,\infty)$ such that $\Phi^{-1}(q) \leq c_1 \sqrt q$ \ for all $q \in [0,1]$. Then write for  $x \geq 1$,
$$
\bar V(x) = x\int_{[0,x]} \frac{u}{x} L(\mathrm du) \leq 3 x \int_{[0,x]} \left(1-e^{- \frac{u}{x}} \right)  L(\mathrm du) \leq 3 x \Phi^{-1}(1/x) \leq 3c_1 \sqrt x
$$
where we have used for the first inequality that $y \leq 3(1-e^{-y})$ for $y \in [0,1]$ and (\ref{LK_invPhi}) for the second inequality.
This leads to the existence of $c_2 \in (0,\infty)$ such that
$$
\bar V(x) \leq c_2 \left( 1 + \sqrt x\right), \quad \forall x \geq 0,
$$
which in turn implies that for all $u,t \geq 0$, $q>0$,
$$
\bar V\left(2 \left(\frac{\sqrt{u^4+4u^2 \sqrt{\Phi(q) t}}}{2 \Phi(q)} \right)\right)  \leq c_2 \left( 1 + \frac{u+\sqrt{2u} \ \Phi(q)^{1/8}t^{1/8}}{\sqrt{\Phi(q)}}\right).
$$
And then, with (\ref{eq:largebound}), we get the upper bound
$$\mathbb E\left[ e^{-q\xi_{\rho(t)}}\right] \leq c_3\left(1+(\Phi(q))^{-1/2}+(\Phi(q))^{-3/8}t^{1/8} \right) \phi(q)e^{-2 \sqrt{\Phi(q)t}}$$
for some $c_3>0$, as expected.

\section{Asymptotics of \texorpdfstring{$\xi_{\rho}$}{xiRho}}
\label{sec:asymptotic_sub}

Theorem~\ref{thm:lapTransform} will allow us to obtain the asymptotic behavior of $\mathbb E[\exp(-q\xi_{\rho(t)})]$ as $t\rightarrow \infty$, both when $q$ is fixed or when $q=q(t)$ is an appropriate function of $t$ chosen to obtain the scaling limit of $\xi_{\rho(t)}$. Indeed, observe that in the expression
\begin{equation}
\label{eq:Laplacebis}
  \E\left[e^{-q \xi_{\rho(t)}}\right]  = \int_0^\infty e^{-\Phi(q)x-\frac{t}{x}} xL(\dd x)
\end{equation}
the function $x \mapsto e^{-\Phi(q)x-\frac{t}{x}}$ attains its maximum at $x=\sqrt{t/\Phi(q)}$, which is equal to $e^{-2 \sqrt{\Phi(q)t}}$. Therefore, as long as $L$ is regular enough that we may apply Laplace's saddle point method, we expect the leading term in the asymptotic behavior of $\E[e^{-q \xi_{\rho(t)}}]$ to be $e^{-2 \sqrt{\Phi(q)t}}$, with a polynomial correction. The first aim of this section is to prove Theorem~\ref{thm:asympBehavior}, i.e. that under conditions \eqref{hyp:rv} and \eqref{hyp:tech}, an equivalent of $\E[e^{-q \xi_{\rho(t)}}]$ as $t \to \infty$ can be computed explicitly. Then, with very similar computations, we will be able to obtain the scaling limit of $\xi_{\rho(t)}$ as $t \to \infty$ and prove Theorem~\ref{thm:asymp_distribution}.

The starting point of the proofs of Theorem~\ref{thm:asympBehavior} and Theorem~\ref{thm:asymp_distribution} rely on the following alternative expression for the Laplace transform (\ref{eq:Laplacebis}). Let us recall the notation $V(\dd x) = x L(\dd x)$, and for $x \geq 0$
\[
\bar{V}(x) =  V([0,x]) \quad \text{and} \quad  \bar{L}(x) :=L(x,{\infty}).
\]

\begin{lemma}
\label{eq:dev}
For $u, t, q >0$, let
\[
  a_{u,t,q} = \frac{u^2+2 \sqrt{\Phi(q)t}}{2\Phi(q)} \quad \text{and} \quad b_{u,t,q} = \frac{\sqrt{u^4+4u^2 \sqrt{\Phi(q)t}}}{2\Phi(q)},
\]
\emph{(}note that $a_{u,t,q}>b_{u,t,q}$\emph{)}\footnote{More precisely, we have $(2\Phi(q) a_{u,t,q})^2 - (2 \Phi(q)b_{u,t,q})^2 = 4 \Phi(q) t$. In particular, $2 \Phi(q)(a_{u,t,q} - b_{u,t,q}) \to \infty$ as long as $4 \Phi(q)t \to \infty$, uniformly in $u$.}.
We then have
\begin{equation*}
  \E\left[ e^{-q\xi_{\rho(t)}}\right]
  = \phi(q) e^{-2 \sqrt{\Phi(q)t}}\int_0^\infty 2u e^{-u^2} \left(  \bar V \left(a_{u,t,q}+b_{u,t,q}\right) -  \bar V \left(a_{u,t,q}-b_{u,t,q}\right) \right)  \dd u.
\end{equation*}
\end{lemma}

\begin{proof}
Since $\Phi(q)x + \frac{t}{x} = 2 \sqrt{\Phi(q)t} + \left(\sqrt{\Phi(q)x}-\sqrt{\frac{t}{x}}\right)^2$, we can rewrite (\ref{eq:Laplacebis}) as
\begin{align*}
  \frac{\E\left[ e^{-q\xi_{\rho(t)}}\right]}{\phi(q) e^{-2 \sqrt{\Phi(q)t}}} &= \int_0^\infty e^{-\left(\sqrt{\Phi(q)x}-\sqrt{\frac{t}{x}}\right)^2} x L(\dd x)\\
  &= \int_0^{\infty} \int_{\left|\sqrt{\Phi(q)x}-\sqrt{\frac{t}{x}}\right|}^{\infty}2ue^{-u^2} \mathrm du \ xL(\mathrm dx)\\
  &=  \phi(q)e^{-2 \sqrt{\Phi(q)t}}\int_0^{\infty} 2ue^{-u^2} \mathrm du \left(\int_0^{\infty} \mathbbm 1_{\Big|\sqrt{\Phi(q)x}-\sqrt{\frac{t}{x}}\Big| \leq u} xL(\mathrm dx)\right),
\end{align*}
by Fubini's theorem. Note then that
\[
   \left\{x \in \R : \left|\sqrt{\Phi(q)x}-\sqrt{\frac{t}{x}}\right| \leq u\right\} = \big[a_{u,t,q}-b_{u,t,q},a_{u,t,q}+b_{u,t,q}\big]
\]
and that the Lebesgue measure of $\big\{u>0:V(\{a_{u,t,q}-b_{u,t,q}\})>0\big\}$ is null to complete the proof.
\end{proof}

\medskip

The function $\bar V$  is regularly varying at $\infty$ under \eqref{hyp:rv} which is important but not sufficient for our purpose, as we have to control the increments of $\bar V$. To that end, we apply a result of Doney and Rivero \cite{DR13} in the next section, which relies on \eqref{hyp:tech}. We then complete the proofs of Theorem~\ref{thm:asympBehavior} in Section~\ref{sec:prooffix} and Theorem~\ref{thm:asymp_distribution} in Section~\ref{sec:proofscaling}.

\subsection{Preliminaries: regular variation and asymptotic behavior of the increments of \texorpdfstring{$\bar V$}{V}}
\label{sec:incrementsofV}

We first study the regular variation of $\bar V$ and $\bar{L}$.

\begin{lemma}
\label{lem:varVL}
If $\phi$ is regularly varying at $0$ with index $\gamma \in [0,1]$, the function $\bar V$ is regularly varying at $\infty$ with index $\frac{\gamma}{\gamma+1} \in [0,1/2]$. More precisely:
\begin{enumerate}
\item[\emph{(i)}] When $\gamma>0$, $\bar L$ is also regularly varying at $\infty$, with index  $\frac{-1}{\gamma+1} \in (-1,-1/2]$, and
$$
\bar L(x) \underset{x \rightarrow \infty}{\sim} \frac{1}{\Gamma\left(\frac{\gamma}{1+\gamma}\right)} \cdot  \Phi^{-1}\left(\frac{1}{x}\right) \underset{x \rightarrow \infty}{\sim}  \frac{\gamma \bar V(x)}{x}
$$
\item[\emph{(ii)}] When $\gamma=0$, $\bar L(x) \ll \bar V(x)/x$ and we still have the above behavior for $\bar V$ when moreover $\kappa=0$:
$$
\bar V(x) \underset{x \rightarrow \infty}{\sim}  x \Phi^{-1}\left(\frac{1}{x}\right)
$$
whereas when $\kappa>0$,
$$
\bar V(x)  \underset{x \rightarrow \infty}{\rightarrow} \frac{1}{\kappa}- \frac{\mathbbm 1_{\{c=0\}}}{\kappa +\pi(0,\infty)}.
$$
\end{enumerate}
\end{lemma}

These results on the links between the asymptotics of the tail of the L\'evy measure of a subordinator and its Laplace exponent are classical. We prove them quickly for the sake of completeness and to integrate the function $\bar V$.
Note that  $x \Phi^{-1}\left(\frac{1}{x}\right) \rightarrow \frac{1}{\kappa}$ as $x \rightarrow \infty$ when $\kappa>0$, so in general the two assertions of (ii) cannot be merge.

\begin{proof}
We rewrite \eqref{LK_invPhi} as
\begin{equation}
  \label{eq:rephrased}
\Phi^{-1}(q)=\frac{q \mathbbm 1_{\{c=0\}}}{\kappa+\pi(0,\infty)}+\int_0^q \left(\int_0^{\infty} e^{-xu}V(\mathrm dx) \right) \mathrm du=\frac{q \mathbbm 1_{\{c=0\}}}{\kappa+\pi(0,\infty)}+q\int_0^{\infty}e^{-qu} \bar L(u) \mathrm du,
\end{equation}
using Fubini's theorem.

We first assume that $\kappa = 0$. In this case, $\Phi(q)/q \to 0$ as $q \to 0$, therefore $\Phi^{-1}(q) \gg q$ when $q \rightarrow 0$. Consequently,
\begin{equation}
\label{eq:miseenplace}
  \Phi^{-1}(q) \underset{q \rightarrow 0}{\sim}  \int_0^q \left(\int_0^{\infty} e^{-xu}V(\mathrm dx) \right) \mathrm du.
\end{equation}
As $\phi$ is regularly varying at $0$ with index $\gamma$, $\Phi^{-1}$ is regularly varying at $0$ with index $(\gamma+1)^{-1}$, therefore so is the integral on the right-hand side. Since the function $u \mapsto \int_0^{\infty} e^{-xu}V(\mathrm dx)$ is monotone (decreasing), it is also regularly varying at $0$, with index $(\gamma+1)^{-1}-1$, by the monotone density theorem (\cite[Theorem 1.7.2b]{BGT}) and
$$
\int_0^{\infty} e^{-xq}V(\mathrm dx)  \underset{q \rightarrow 0}{\sim} \frac{q^{-1}}{\gamma+1} \int_0^q \left(\int_0^{\infty} e^{-xu}V(\mathrm dx) \right) \mathrm du  \underset{q \rightarrow 0}{\sim} \frac{q^{-1}}{\gamma+1} \Phi^{-1}(q).
$$
It is then sufficient to use Karamata's Tauberian theorem (\cite[Theorem 1.7.1]{BGT}) to deduce from this the regular variation of $\bar V$ and the equivalence settled in the statements (i) and (ii) of the lemma, regarding $\bar V$.

The behavior of $\bar{L}$ is obtained in a similar fashion. Using \eqref{eq:rephrased} and that $\Phi^{-1}(q) \gg q$ as $q \to 0$, we have
\[
  \Phi^{-1}(q)\underset{q \rightarrow 0}{\sim}  q \int_0^{\infty}e^{-qu} \bar{L}(u) \dd u.
\]
Hence, applying Karamata's tauberian theorem we obtain that
$$
 \int_0^{x} \bar L(u) \mathrm du \underset{x \rightarrow \infty}{\sim} \frac{1}{\Gamma\left(1+\frac{\gamma}{1+\gamma}\right)} \cdot  x \Phi^{-1}(1/x),
$$
those two functions being regularly varying with index $\gamma/(1+\gamma)$. And then applying the monotone density theorem to get the remaining parts of assertions (i) and (ii), still when $\kappa=0$.

We now turn to the case $\kappa > 0$. In this situation
\[
  \lim_{x \to \infty} \bar{V}(x) = \int_0^\infty y L(y) \dd y =\frac{1}{\kappa}- \frac{\mathbbm 1_{\{c=0\}}}{\kappa +\pi(0,\infty)},
\]
as noticed at the end of Section~\ref{sec:L}.
\end{proof}

To get some information on the asymptotic behavior of the increments of $\bar V$, or equivalently of $\bar L$, we recall from Section~\ref{sec:Laplace} that the measure $L$ can be interpreted as the image measure of the excursion measure, noted $n$, of a spectrally negative Lévy process $X$ with Laplace exponent $\Phi$ reflected below its maximum by the application $\zeta$ associating to an excursion its lifetime. In other words, we have $\bar L(x)=n(\zeta>x)$. We can therefore apply a result of Doney and Rivero \cite{DR13,DR16_err} to obtain the following local properties.

\begin{lemma}
\label{lem:increments}
Assume \eqref{hyp:rv} and \eqref{hyp:tech}.
\begin{enumerate}
\item[\emph{(i)}] For any $\Delta(x)>0$ such that  $\Delta(x)=o(x)$ as $x \rightarrow \infty$,
$$
\sup_{\Delta \in (0,\Delta(x)]} \left|\frac{x}{\bar L(x)} \cdot \frac{\bar L(x)-\bar L(x+\Delta)}{\Delta} - \frac{1}{1+\gamma}\right| \underset{x \rightarrow \infty}\rightarrow 0.
$$
\item[\emph{(ii)}] There exists $x_0>0$ such that
$$
\sup_{\Delta>0} \left|\bar L(x)-\bar L(x+\Delta) \right| \leq \frac{2 \Delta \bar L(x)}{x} \qquad \text{for all } x \geq x_0.
$$
\end{enumerate}
\end{lemma}

\begin{proof}
By Theorem 1 (iii) in \cite{DR13} and the erratum \cite{DR16_err}, under the hypotheses \eqref{hyp:rv} and \eqref{hyp:tech},
$$
\sup_{\Delta \in (0,\Delta_0]} \left|\frac{x}{\bar L(x)} \cdot \frac{\bar L(x)-\bar L(x+\Delta)}{\Delta} - \frac{1}{1+\gamma}\right| \underset{x \rightarrow \infty}\rightarrow 0
$$
for any fixed $\Delta_0>0$ (note that with the notation $\alpha,\rho$ of \cite{DR13}, we have here $\alpha=1+\gamma$ and $\bar \rho=\frac{1}{1+\gamma}$).
Now fix $\varepsilon>0$ and let $x_{\varepsilon}>0$ be such that for $x \geq x_{\varepsilon}$
$$
\sup_{\Delta \in (0,1]} \left|\frac{x}{\bar L(x)} \cdot \frac{\bar L(x)-\bar L(x+\Delta)}{\Delta} - \frac{1}{1+\gamma}\right| \leq \varepsilon.
$$
Then for any $\Delta>0$ and all  $x \geq x_{\varepsilon}$
\begin{eqnarray*}
\bar L(x)-\bar L(x+\Delta) &=& \sum_{k=0}^{\lfloor\Delta \rfloor-1} \left(\bar L(x+k)-\bar L(x+k+1)\right)+\bar L(x+\lfloor\Delta \rfloor)-\bar L(x+\Delta) \\
& \leq& \left( \frac{1}{1+\gamma} + \varepsilon\right)  \left( \sum_{k=0}^{\lfloor\Delta \rfloor-1}  \frac{\bar L(x+k)}{x+k} +  \frac{\bar L(x+\lfloor\Delta \rfloor)}{x+\lfloor\Delta \rfloor} \cdot (\Delta-\lfloor\Delta \rfloor)\right) \\
& \leq& \left( \frac{1}{1+\gamma} + \varepsilon\right) \frac{\bar L(x)}{x} \cdot \Delta
\end{eqnarray*}
where for the last inequality we use that $x \mapsto \bar L(x)/x$ is decreasing. This implies the point (ii) of the lemma and is a first step in the proof of the point (i). To complete the proof of the uniform convergence on intervals of the form $(0,\Delta(x)]$ with $\Delta(x)=o(x)$, we note, similarly as above, that for any $\Delta \in (0,\Delta(x)]$ and all  $x \geq x_{\varepsilon}$,
\begin{eqnarray*}
\bar L(x)-\bar L(x+\Delta) &\geq&  \left( \frac{1}{1+\gamma} - \varepsilon\right)  \left( \sum_{k=0}^{\lfloor\Delta \rfloor-1}  \frac{\bar L(x+k)}{x+k} +  \frac{\bar L(x+\lfloor\Delta \rfloor)}{x+\lfloor\Delta \rfloor} \cdot (\Delta-\lfloor\Delta \rfloor)\right) \\
 &\geq&
 \left( \frac{1}{1+\gamma} - \varepsilon\right) \frac{\bar L(x+\Delta(x))}{x+\Delta(x)} \cdot \Delta.
\end{eqnarray*}
Now, since the function $x \mapsto \bar L(x)/x$ is regularly varying as $x \rightarrow \infty$ with index $-\frac{1}{1+\gamma}-1$, by the Uniform Convergence Theorem (\cite[Theorem 1.5.2]{BGT}), for all $x$ large enough
$$
 \frac{\bar L(x+\Delta(x))}{x+\Delta(x)} \cdot \frac{x}{\bar L(x)} \geq \frac{1}{\left( 1+\frac{\Delta(x)}{x}\right)^{\frac{1}{1+\gamma}+1}}-\varepsilon.
$$
Taking $x$ larger if necessary, we may assume that this last lower bound is in turn larger than $1-2 \varepsilon$  since $\Delta(x)=o(x)$. So finally we have that for all $x$ large enough and then all $\Delta \in (0,\Delta(x)]$,
$$
\bar L(x)-\bar L(x+\Delta) \geq  \left( \frac{1}{1+\gamma} - \varepsilon\right) (1-2\varepsilon)  \frac{\bar L(x)}{x} \cdot \Delta
$$
as expected.
\end{proof}

This readily yields the following uniform behavior of the increments of $V$.

\newpage 

\begin{corollary}
\label{cor:increments}
Assume \eqref{hyp:rv} and \eqref{hyp:tech}.
\begin{enumerate}
\item[\emph{(i)}]
For any $\Delta(x)>0$ such that  $\Delta(x)=o(x)$ as $x \rightarrow \infty$,
$$
\sup_{\Delta \in (0,\Delta(x)]} \left|\frac{\bar V(x+\Delta)-\bar V(x)}{\Delta \bar L(x)} - \frac{1}{1+\gamma}\right| \underset{x \rightarrow \infty}\rightarrow 0.
$$
\item[\emph{(ii)}] There exists $x_1>0$ such that
$$
\sup_{\Delta>0} \left|\bar V(x)-\bar V(x+\Delta) \right| \leq 2 \Delta \bar L(x) +2 \frac{\bar L(x)}{x} \Delta^2 \quad \text{for all } x \geq x_1.
$$
\end{enumerate}
\end{corollary}

\begin{proof}
Recall that $V(\mathrm dx)=x L(\mathrm dx)$, so that
\begin{equation}
\label{eq:VL}
\bar V(x+\Delta)-\bar V(x)=x\left( \bar L(x)-\bar L(x+\Delta)\right)+\int_0^{\Delta} \left(\bar L(x+y)- \bar L(x+\Delta)\right) \mathrm dy.
\end{equation}
From Lemma~\ref{lem:increments} (i), we have on the one hand that
$$
\sup_{\Delta \in (0,\Delta(x)]} \left|\frac{x\left( \bar L(x)-\bar L(x+\Delta)\right)}{\Delta \bar L(x)}- \frac{1}{1+\gamma}\right| \underset{x \rightarrow \infty}\rightarrow 0
$$
and on the other hand that for $x$ large enough and then all $\Delta \in (0,\Delta(x)]$
\begin{eqnarray*}
\left|\int_0^{\Delta} \left(\bar L(x+y)- \bar L(x+\Delta)\right) \mathrm dy \right| &\leq&  \frac{2}{1+\gamma} \int_0^{\Delta} \frac{\bar L(x+y)}{x+y} (\Delta-y) \mathrm dy \\
&\leq &  \frac{1}{1+\gamma}  \cdot \frac{\bar L(x)}{x} \cdot \left(\Delta\right)^2,
\end{eqnarray*}
hence
$$
\sup_{\Delta \in (0,\Delta(x)]} \frac{\left|\int_0^{\Delta} \left(\bar L(x+y)- \bar L(x+\Delta)\right) \mathrm dy \right|}{\Delta \bar L(x)} \leq  \frac{1}{1+\gamma}  \cdot \frac{\Delta(x)}{x} \longrightarrow 0.
$$
This gives (i). The upper bound in (ii) is obtained in a similar way, by combining Lemma~\ref{lem:increments}(ii) with \eqref{eq:VL}.
\end{proof}

\subsection{Asymptotics of \texorpdfstring{$\mathbb E[\exp(-q\xi_{\rho(t)})]$ when $t \rightarrow \infty$}{the Laplace transform at large times}}
\label{sec:prooffix}

In this section $q>0$ is fixed and we let $t \rightarrow \infty$. We have set up all the key elements necessary for the proof of Theorem~\ref{thm:asympBehavior}, which ends as follows.

\begin{proof}[\bf Proof of Theorem~\ref{thm:asympBehavior}.] Set for $u,t,q>0$,
$$
I_{u,t,q}:=\bar V \left(a_{u,t,q}+b_{u,t,q}\right)-\bar V \left(a_{u,t,q}-b_{u,t,q}\right),
$$
with the notation $a,b$ of Lemma~\ref{eq:dev},
so that
$$
\mathbb E\left[e^{-q\xi_{\rho(t)}}\right]=\phi(q)e^{-2 \sqrt{\Phi(q)t}}\int_0^{\infty} 2ue^{-u^2} I_{u,t,q} \mathrm du.
$$
Then first note that for each fixed $u,q>0$,
$$
a_{u,t,q} \underset{t \rightarrow \infty}{\sim} \frac{\sqrt t}{\sqrt{\Phi(q)}} \quad \text{and} \quad b_{u,t,q} \underset{t \rightarrow \infty}{\sim}  \frac{u t^{1/4}}{(\Phi(q))^{3/4}}
$$
and that Corollary~\ref{cor:increments} (i) yields
$$
I_{u,t,q} \underset{t \rightarrow \infty}{\sim} \frac{2 b_{u,t,q} \bar L(a_{u,t,q}-b_{u,t,q})}{1+\gamma}.
$$
Together with the regular variation of $\bar L$ at $\infty$ (Lemma~\ref{lem:varVL} (i)), we get that
$$\frac{I_{u,t,q}}{t^{1/4} \bar L(\sqrt t)} \underset{t\rightarrow \infty} \longrightarrow \frac{2u}{1+\gamma} \cdot \Phi(q)^{\frac{1}{2(1+\gamma)}-\frac{3}{4}}.$$
It remains to conclude with the dominated convergence converge theorem that this implies
\begin{equation}
\label{cv:DCT}
\frac{\int_0^{\infty}2ue^{-u^2} I_{u,t,q}\mathrm du}{t^{1/4} \bar L(\sqrt t)} \underset{t\rightarrow \infty} \longrightarrow  \frac{1}{1+\gamma} \cdot \Phi(q)^{\frac{1}{2(1+\gamma)}-\frac{3}{4}} \int_0^{\infty}4u^2e^{-u^2}\mathrm du.
\end{equation}
The convergence indeed gives Theorem~\ref{thm:asympBehavior}  since $\Gamma(\frac{\gamma}{\gamma+1})L(\sqrt t) \underset{t \rightarrow \infty}\sim \Phi^{-1}(1/\sqrt t)$ by Lemma~\ref{lem:varVL} (i) and
 $\int_0^{\infty} 4u^2e^{-u^2}\mathrm du=\sqrt \pi$.

\bigskip

To apply the dominated convergence theorem and get \eqref{cv:DCT}, we need to split the term under the integral into several pieces. In that aim, note that for all $A>0$, there exists $t(A)>0$ (that may depend on $q$) such that for all $t \geq t(A)$ and then all $u\leq \frac{\sqrt t}{2\sqrt{A}}$
$$
a_{u,t,q}-b_{u,t,q}=\frac{t}{(a_{u,t,q}+b_{u,t,q})  \Phi(q)} \geq A
$$
(to see this note that the fraction in the left-hand side is decreasing in $u$). We denote by
$$E(t,A)=\left\{u>0: a_{u,t,q}-b_{u,t,q} \geq A \right\}.$$
We also need the following consequence of the regular variation of $L$ with index $-1/(1+\gamma)$: there exists $x_2>0$ such that for all $x \geq y  \geq x_2$,
$$
\frac{\bar L(y)}{\bar L(x)} \leq 2 \left(\frac{x}{y}\right)^{\frac{2}{(1+\gamma)}},
$$
applying the classical Potter's bounds for regularly varying functions (\cite[Theorem 1.5.6]{BGT}).
So, with the notation $x_1$ of Corollary~\ref{cor:increments} (ii), if $t \geq t(\max(x_1,x_2))$ and $u \in E(t,\max(x_1,x_2))$, we have that
\begin{eqnarray*}
\left|I_{u,t,q}\right| &\underset{\text{by Cor.~\ref{cor:increments} (ii)}}\leq& 4 b_{u,t,q} \bar L(a_{u,t,q}-b_{u,t,q}) \left( 1+ \frac{2b_{u,t,q}}{a_{u,t,q}-b_{u,t,q}} \right) \\
&\underset{\text{by Potter's bounds}}\leq& 4 b_{u,t,q}  \left( 1+ \frac{2\Phi(q)b_{u,t,q}(a_{u,t,q}+b_{u,t,q})}{t} \right) \\
&& \qquad \quad \times 2 \ \bar L \left({\frac{\sqrt t}{\sqrt{\Phi(q)}}}\right) \left({\frac{ \sqrt{\Phi(q)}}{\sqrt t}}\left(a_{u,t,q}+b_{u,t,q}\right)\right)^{\frac{2}{(1+\gamma)}}
\end{eqnarray*}
(we have also used that $(a_{u,t,q}-b_{u,t,q})(a_{u,t,q}+b_{u,t,q})=t/\Phi(q)$). Taking $t$ larger if necessary, we also have that $$ \bar L\left({\frac{\sqrt t}{\sqrt{\Phi(q)}}}\right)\leq \bar L(\sqrt t) \left(2(\Phi(q))^{\frac{2}{(1+\gamma)}} \wedge 1\right)$$ (using the Potter's bounds when $\Phi(q)\geq 1$ and that $\bar L$ is decreasing when  $\Phi(q)\leq 1$).
Using the definitions of $a_{u,t,q},b_{u,t,q}$, we then deduce from these bounds that for $t$ large enough and all $u$ in $E(t,\max(x_1,x_2))$,
\begin{equation*}
\label{majo:I}
\left|I_{u,t,q}\right|\leq  c(q) t^{1/4} \bar L\big({\sqrt t}\big) P(u)\left(u^2+u+1\right)^{\frac{2}{1+\gamma}},
\end{equation*}
where $P$ is a polynomial and $c(q)$ does not depend on $t$ and $u \in E(t,\max(x_1,x_2))$. We can therefore apply the dominated convergence theorem and get that $$\frac{\int_0^{\infty} 2ue^{-u^2} I_{u,t,q} \mathbbm 1_{u \in E(t,\max(x_1,x_2))}\mathrm du}{t^{1/4} \bar L(\sqrt t)} \underset{t\rightarrow \infty} \longrightarrow  \frac{1}{1+\gamma} \cdot \Phi(q)^{\frac{1}{2(1+\gamma)}-\frac{3}{4}} \int_0^{\infty}4u^2e^{-u^2}\mathrm du.$$ It remains to show that $\int_0^{\infty} 2ue^{-u^2} I_{u,t,q} \mathbbm 1_{u \notin E(t,\max(x_1,x_2))}\mathrm du=o\big({t^{1/4} \bar L(\sqrt t)}\big)$. In that aim we write for $u \notin E(t,\max(x_1,x_2))$
\begin{eqnarray*}
I_{u,t,q}&=& \bar V \left(a_{u,t,q}+b_{u,t,q}\right)-\bar V(\max(x_1,x_2))+\bar V(\max(x_1,x_2))-\bar V \left(a_{u,t,q}-b_{u,t,q}\right) \\
& \leq&  \bar V \left(a_{u,t,q}+b_{u,t,q}\right)-\bar V(\max(x_1,x_2))+\bar V(\max(x_1,x_2)) \\
&{\leq}&  2 \bar L(\max(x_1,x_2)) \left(a_{u,t,q}+b_{u,t,q}-\max(x_1,x_2)+\frac{(a_{u,t,q}+b_{u,t,q}-\max(x_1,x_2))^2}{\max(x_1,x_2)}  \right) \\
&\leq& d(q) \left( u^4+u^2 \sqrt t +1\right)
\end{eqnarray*}
for some $d(q)$ independent of $t>0, u \notin E(t,\max(x_1,x_2))$, where we used Corollary~\ref{cor:increments} (ii) for the penultimate inequality, and that $u \notin E(t,\max(x_1,x_2))$ implies $a_{u,t,q}-b_{u,t,q}\leq \max(x_1,x_2)$ and the definition of $b_{u,t,q}$ for the last inequality.
Finally, since $u \notin E(t,\max(x_1,x_2))$ implies that  $u>\sqrt t/2\sqrt{(\max(x_1,x_2))}$,
\begin{eqnarray*}
 \frac{\int_0^{\infty} 2ue^{-u^2} \mathbbm 1_{u \notin E(t,\max(x_1,x_2))}I_{u,t,q} \mathrm du}{t^{1/4} \bar L(\sqrt t)} =O \left( \frac{e^{-\frac{t}{4(\max(x_1,x_2))}}}{t^{1/4} \bar L(\sqrt t)} \right)  \underset{t\rightarrow \infty} \longrightarrow 0
\end{eqnarray*}
as required.
\end{proof}

\subsection{Convergence in distribution of \texorpdfstring{$\Phi^{-1}(1/t)\xi_{\rho(t)}$}{the tagged fragment}}
\label{sec:proofscaling}

This section is devoted to the proof of Theorem~\ref{thm:asymp_distribution}, i.e. the scaling limit of $\xi_{\rho(t)}$ as $t\rightarrow \infty$. This is done using Theorem~\ref{thm:lapTransform} to study the joint asymptotic behavior of the Laplace transform of $\xi_\rho(t)$ as $t \to \infty$ and $q \to 0$.

\subsubsection{Proof of Theorem~\ref{thm:asymp_distribution} (i)}

In this part, we only assume (\ref{hyp:rv}). The convergence in distribution of $\Phi^{-1}(1/t)\xi_{\rho(t)}$ to $D(\gamma)$ could certainly be proved by using that under  (\ref{hyp:rv}) $\xi$ is in the domain of attraction of a stable subordinator and the definition of the time change $\rho$. We present here an alternative proof, which is based on Theorem~\ref{thm:lapTransform}. We use again the rewriting of Lemma~\ref{eq:dev} which gives for $\lambda,t>0$
\begin{equation}
\label{eq:scaling}
\mathbb E\left[e^{-\lambda \Phi^{-1}(1/t)\xi_{\rho(t)}} \right]=\phi\left(\lambda \Phi^{-1}(1/t)\right)e^{-2 \sqrt{\Phi\left(\lambda \Phi^{-1}(1/t)\right)t}}\int_0^{\infty} 2ue^{-u^2}J_{u,t} \mathrm du
\end{equation}
with
$$
J_{u,t}=\bar V\left(a_{u,t,\lambda  \Phi^{-1}(1/t)}+b_{u,t,\lambda  \Phi^{-1}(1/t)}\right)- \bar V\left(a_{u,t, \lambda \Phi^{-1}(1/t)}-b_{u,t, \lambda \Phi^{-1}(1/t)}\right)
$$
with the notation $a,b$ of Lemma~\ref{eq:dev}.

In what follows $\lambda$ is fixed. Under the regular variation hypothesis  (\ref{hyp:rv}) we clearly have that when $t \rightarrow \infty$,
$
\Phi\left(\lambda \Phi^{-1}(1/t)\right)t \rightarrow \lambda^{\gamma+1}
$
and $$
\phi\left(\lambda \Phi^{-1}(1/t)\right)\ \sim \ (\gamma+1)\ \frac{\Phi\left(\lambda \Phi^{-1}(1/t)\right)}{\lambda \Phi^{-1}(1/t)} \ \sim \ \frac{(\gamma+1)  \lambda^{\gamma} }{ t\Phi^{-1}(1/t)} \ \sim \ \frac{(\gamma+1)  \lambda^{\gamma} }{\gamma \Gamma \left(\frac{\gamma}{\gamma+1} \right)} \cdot  \bar V(t)
$$
the last equivalence being a consequence of Lemma~\ref{lem:varVL} (i).  Using further the regular variation of $\bar V$, we note that
\begin{equation*}
J_{u,t} \ \sim \ A(u,\lambda)   \cdot  \bar V(t)
\end{equation*}
where
\begin{equation*}
A(u,\lambda):=\left(\frac{u^2+2 \lambda^{\frac{\gamma+1}{2}}+u \sqrt{u^2+4 \lambda^{\frac{\gamma+1}{2}}}}{2 \lambda ^{\gamma+1}} \right)^{\frac{\gamma}{\gamma+1}} - \left(\frac{u^2+2 \lambda^{\frac{\gamma+1}{2}}-u \sqrt{u^2+4 \lambda^{\frac{\gamma+1}{2}}}}{2 \lambda ^{\gamma+1}} \right)^{\frac{\gamma}{\gamma+1}},
\end{equation*}
and also that according to the Potter's bounds (\cite[Theorem 1.5.6]{BGT})
$$
 \frac{J_{u,t}}{\bar V(t)} \leq  \frac{2\gamma}{1+\gamma} \left(\frac{u^2+2 \lambda^{\frac{\gamma+1}{2}}+u \sqrt{u^2+4 \lambda^{\frac{\gamma+1}{2}}}}{2\lambda ^{\gamma+1}} \right)^{\frac{\gamma}{\gamma+1}}
$$
for all $t$ large enough and then all $u>0$. We can therefore apply the dominated convergence theorem and obtain from (\ref{eq:scaling}) that
\begin{equation*}
\lim_{t \to \infty} \mathbb E\left[e^{-\lambda \Phi^{-1}(1/t)\xi_{\rho(t)}} \right] \ = \ \frac{(\gamma+1)\lambda^{\gamma} }{\gamma \Gamma  \left(\frac{\gamma}{\gamma+1} \right)}  \cdot  e^{-2 \lambda^{\frac{\gamma+1}{2}}}   \cdot  \int_0^{\infty} 2ue^{-u^2} A(u,\lambda) \mathrm du.
\end{equation*}
We conclude by noticing that this last expression is the Laplace transform $\mathbb E[e^{-\lambda \xi_1^{(\gamma)}}]$ when  $\xi^{(\gamma)}$ is a subordinator with Laplace exponent $q \mapsto (\gamma+1)q^{\gamma}$, i.e. the Laplace transform of the distribution $D(\gamma)$. This is easily seen by using the expression of the L\'evy measure $L$ associated to $\xi^{(\gamma)}$, see the paragraph preceding Corollary~\ref{cor:stable}, and Lemma~\ref{eq:dev}.

\subsubsection{Proof of Theorem~\ref{thm:asymp_distribution} (ii)}

We assume now that \eqref{hyp:rv} holds with $\gamma=1$. The almost sure convergence when $m=c+\int_0^{\infty}x\pi(\mathrm dx)$ is finite follows trivially from the almost sure convergence of $t^{-1}\xi_t$ to $m$ and the definition (\ref{eq:time_change}) of $\rho$ which then yields $m\rho(t)^2/2 \sim t$ as $t \to \infty$.

Assuming additionally that $a=\int_0^{\infty}x^2\pi(\mathrm dx)<\infty$, we want to prove a central limit theorem. We first note that under this hypothesis:
\begin{equation}
\label{equiv:Phi}
\phi(q)\underset{q \rightarrow 0}{=}mq-\frac{a}{2} q^2+o(q^2), \qquad \Phi(q)\underset{q \rightarrow 0}{=}\frac{m}{2}q^2-\frac{a}{3!}q^3 +o(q^3), \qquad \Phi^{-1}(q)\underset{q \rightarrow 0}{\sim } \sqrt{\frac{2}{m}q}.
\end{equation}
By Lemma~\ref{eq:dev}, for all $\lambda>0$ and all $t>0$
\begin{eqnarray*}
\mathbb E\left[e^{-\lambda t^{1/4}\left(\frac{\xi_{\rho(t)}}{\sqrt t} -\sqrt{2m} \right)} \right] = e^{\lambda t^{1/4}\sqrt{2m}} \phi\big(\lambda t^{-1/4}\big) e^{-2 \sqrt{\Phi\left(\lambda t^{-1/4}\right)t}} \int_0^{\infty}2ue^{-u^2} K_{u,t} \mathrm du
\end{eqnarray*}
where
$$
K_{u,t}:=\bar V\left(a_{u,t,\lambda t^{-1/4}}+b_{u,t,\lambda t^{-1/4}}\right)- \bar V\left(a_{u,t,\lambda t^{-1/4}}-b_{u,t,\lambda t^{-1/4}}\right).
$$
Using the asymptotic expansion at 0 of $\Phi$ and $\sqrt{1-x}=1- \frac{x}{2}+o(x)$ as $x \rightarrow 0$, we immediately get that
$$
e^{\lambda t^{1/4}\sqrt{2m}}e^{-2 \sqrt{\Phi\left(\lambda t^{-1/4}\right)t}} \ \underset{t \rightarrow \infty}{\longrightarrow} \ e^{\frac{a}{3 \sqrt{2m}} \lambda^2},
$$
where we recognize in the right-hand side the Laplace transform   of a centered gaussian distribution with variance $\frac{\sqrt 2a}{3 \sqrt{m}}$.
It remains to prove that
\begin{equation}
\label{conc}
\phi\left(\lambda t^{-1/4}\right) \int_0^{\infty}2ue^{-u^2}  K_{u,t} \mathrm du   \ \underset{t \rightarrow \infty}{\longrightarrow} 1.
\end{equation}
We proceed as in the proof of Theorem~\ref{thm:asympBehavior}, relying on Corollary~\ref{cor:increments}. First, note that by (\ref{equiv:Phi}), for each fix $u>0$,
$$
a_{u,t,\lambda t^{-1/4}} \ \underset{t \rightarrow \infty}{\sim} \sqrt{\frac{2}{m}} \cdot \frac{t^{3/4}}{\lambda} \quad \text{and} \quad  b_{u,t,\lambda t^{-1/4}} \ \underset{t \rightarrow \infty}{\sim} \ u \left(\frac{2}{m}\right)^{3/4} \cdot \frac{t^{5/8
}}{ \lambda^{3/2}}.
$$
Consequently, for each fixed $u>0$, by Corollary~\ref{cor:increments} (i) (since $b_{u,t,\lambda t^{-1/4}} =o(a_{u,t,\lambda t^{-1/4}} )$  and $\gamma=1$)
\begin{eqnarray*}
K_{u,t} & \underset{t \rightarrow \infty}{\sim} & b_{u,t,\lambda t^{-1/4}} \bar L\big(a_{u,t,\lambda t^{-1/4}} -b_{u,t,\lambda t^{-1/4}} \big) \\
& \underset{t \rightarrow \infty}{\sim} &  b_{u,t,\lambda t^{-1/4}} \frac{\sqrt 2}{\sqrt{m \pi a_{u,t,\lambda t^{-1/4}}}}  \\
 & \underset{t \rightarrow \infty}{\sim} & \frac{u}{\sqrt \pi} \cdot \frac{2}{m}  \cdot \frac{t^{1/4}}{\lambda}
\end{eqnarray*}
where in the second line we used Lemma~\ref{lem:varVL} (i) and (\ref{equiv:Phi}) to get that
$$
\bar L(x) \underset{x \rightarrow \infty} \sim \frac{1}{\Gamma(1/2)} \cdot \Phi^{-1}\left(\frac{1}{x}\right) \underset{x \rightarrow \infty} \sim  \frac{\sqrt 2}{\sqrt{m \pi x}} .
$$
If we could apply the dominated convergence theorem, this would indeed imply (\ref{conc}) since $\phi\big(\lambda t^{-1/4}\big) \sim m \lambda t^{-1/4}$ and $\int_0^{\infty}2u^2e^{-u^2}\mathrm du= \sqrt \pi /2$. The proof to dominate appropriately \linebreak $\phi\big(\lambda t^{-1/4}\big)K_{u,t}$ holds in a similar way as we did in the proof of Theorem~\ref{thm:asympBehavior}, by splitting according to whether $u \geq a \sqrt t$ or $u \leq a \sqrt t$ for some appropriate constant $a$. We leave the details to the reader.

\section{Asymptotics of the whole fragmentation process}
\label{sec:wholefrag}

Let $F$ be a fragmentation process with characteristics $(\nu,c, |\log|^{-1})$ and recall the notations  $$\phi_{(\nu,c)}(q)=cq+\int_{\mathcal S} \sum_{i\geq 1}s_i(1-s_i^q) \nu(\mathrm d \mathbf s),$$
$\Phi_{(\nu,c)}$ for the primitive of $\phi_{(\nu,c)}$ null at 0 and $\Phi^{-1}_{(\nu,c)}$ for the inverse of $\Phi_{(\nu,c)}$. The computations of the previous section show that the logarithm of the size of a typical fragment of the process $F$ at a large time $t$ will typically be of order $-1/\Phi_{(\nu,c)}^{-1}(1/t)$. We thus take interest in the following random probability point measure on $(0,\infty)$, which captures the makeup of the population at a large time
\[
  \mathcal{E}_t := \sum_{i \geq 1} F_i(t) \delta_{\Phi_{(\nu,c)}^{-1}(1/t) |\log F_i(t)|}.
\]
Our goal is to prove the asymptotics of this measure settled in Proposition~\ref{thm:cvdps} and for this we follow the same strategy as that of the proof of \cite[Theorem 1]{BertoinAB03} by considering one and then two typical fragments.
It is worth noting that Proposition~\ref{prop:taggedFragment} and Theorem~\ref{thm:asymp_distribution} (i) immediately imply that for all continuous bounded functions $f:(0,\infty) \rightarrow \mathbb R$, when $\phi_{(\nu,c)}$ is regularly varying at 0 with index $\gamma \in (0,1]$,
\begin{equation}
\label{eq:intermediaire}
  \E\left[ \int_0^{\infty} f \dd \mathcal{E}_t \right] = \E\left[ f\left(\Phi_{(\nu,c)}^{-1}(1/t)\xi_{\rho^{(\nu,c)}(t)} \right) \right]  \ {\underset{t \to \infty}\longrightarrow} \ \int_0^{\infty} f \dd D(\gamma).
\end{equation}
Therefore, to complete the proof of Proposition~\ref{thm:cvdps} (i), it will be enough to show that $\int_0^{\infty} f \dd \mathcal{E}_t$ is well-concentrated around its mean. To do so, we study the asymptotics of the second moment of this quantity.
\begin{lemma}
\label{lem:secondMoment}
Assume that $\phi_{(\nu,c)}$ is regularly varying at 0 with index $\gamma \in (0,1]$. Then for all continuous bounded functions $f:(0,\infty) \rightarrow \mathbb R$
\[
\E\left[ \left(\int_0^{\infty} f \dd \mathcal{E}_t\right)^2 \right]  \ {\underset{t \to \infty}\longrightarrow} \ \left( \int_0^{\infty} f \dd D(\gamma) \right)^2.
\]
\end{lemma}

Using this lemma, the proof Proposition~\ref{thm:cvdps} (i) follows straightforwardly.

\begin{proof}[\bf Proof of Proposition~\ref{thm:cvdps} (i)]
Using Lemma~\ref{lem:secondMoment} together with (\ref{eq:intermediaire}),
we see that $\int_0^{\infty} f \dd \mathcal{E} \to \int_0^{\infty} f \dd D(\gamma)$ in $L^2$. Applying this result to $x \mapsto e^{i a x}$, we deduce that the Fourier transform of $\mathcal{E}_t$ converges pointwise in probability to the Fourier transform of $D(\gamma)$. As a result, $\mathcal{E}_t$ converges pointwise in probability to $D(\gamma)$.
\end{proof}

We now turn to the proof of Lemma~\ref{lem:secondMoment}.

\begin{proof}[\bf Proof of Lemma~\ref{lem:secondMoment}]
We use the construction of the fragmentation from its homogeneous counterpart as detailed in Section~\ref{sec:intro_frag}. Consider $(I^\h_j(t), j \geq 1)$ an interval version of this homogeneous fragmentation with dislocation measure $\nu$ and erosion coefficient $c$. Let $U$ and $U'$ be two independent uniformly distributed on $(0,1)$ random variables, independent of $I^\h$, and define
\[
  \xi_t = -\log |I^\h_{t,U}| \quad \text{and}\quad \xi'_t = -\log |I^\h_{t,U'}|,
\]
where $I^\h_{t,U}$, respectively $I^\h_{t,U'}$, stands from the unique interval at time $t$ containing $U$, resp. $U'$. We note that $\xi$ and $\xi'$ are two subordinators with Laplace exponent $\phi{(\nu,c)}$, which are \emph{not} independent. However, writing
\[
  T  := \inf\left\{t > 0 : I^\h_{t,U} \neq I^\h_{t,U'}\right\},
\]
we note that $(\xi_{T+t} - \xi_T,t\geq 0)$ and $(\xi'_{T+t} - \xi_T',t\geq 0)$ are i.i.d. subordinators, by the strong Markov property. It is then a simple computation to remark, with transparent notation, that
\[
  \E\left[ \left(\int_{0}^{\infty} f \dd \mathcal{E}_t\right)^2 \right]  = \E\left[ f\left(\Phi_{(\nu,c)}^{-1}(1/t) \xi_{\rho(t)}) f(\Phi_{(\nu,c)}^{-1}(1/t) \xi'_{\rho'(t)}\right)\right].
\]
Since $\rho(T)=\rho'(T)$ is an a.s. finite stopping time and $\Phi_{(\nu,c)}^{-1}(1/t) \to 0$, we conclude that $\Phi_{(\nu,c)}^{-1}(1/t) \xi_{\rho(t)}$ and $\Phi_{(\nu,c)}^{-1}(1/t) \xi'_{\rho'(t)}$ are asymptotically independent, yielding
\[
  \E\left[ \left(\int_0^{\infty} f \dd \mathcal{E}_t\right)^2 \right]  \ {\underset{t \to \infty}\longrightarrow} \ \left( \int_0^{\infty} f \dd D(\gamma) \right)^2,
\]
which completes the proof.
\end{proof}

To finish, the proof of Proposition~\ref{thm:cvdps} (ii) holds very similarly by combining Proposition~\ref{prop:taggedFragment} and Theorem~\ref{thm:asymp_distribution} (ii) with a concentration result which is proved by using 2 typical fragments. This is left to the reader.

\section{Comparison with the convex hull asymptotic}
\label{sec:discussions}

We return to our initial motivation and finish with some informal thoughts, notably regarding the observed discrepancy between our result in Proposition~\ref{prop:unifBound} and the prediction in \cite{DBBM22}. Indeed, recall that the physicists predicted that
\[
  2\pi - \E[\bar{P}_t] \underset{t \to \infty}{\sim} c t^{1/4} e^{-2t^{1/2}},
\]
while we show that for the fragmentation toy-model, choosing $\Phi$ so that $\Phi(2)=1$, we have
\[
\mathbb E \Bigg[ \sum_{i \geq 1} F_i(t)^3 \Bigg]\leq C\big(1+ t^{1/8}\big) e^{-2t^{1/2}}.
\]
Several factors might explain the difference in the asymptotic behavior of the convex hull of the Brownian motion and our simplified model, we list here some potential explanations.

\begin{description}
  \item[Exponential distribution in the narrow escape problem.] We use the approximation \eqref{eqn:narrowEscapeApproximation} in the narrow escape problem to approach the first hitting time of an interval of length $\epsilon$ by $B^{\D}$ by an exponential random variable with parameter $|\log \epsilon|$. It is possible that using a different rate of division might lead us to the correct asymptotic. However, note that this hypothesis is identical to the one made in \cite{DBBM22} to obtain \eqref{equi:DBBM22}.
  \item[Correlations between the hitting times of distinct intervals.] Contrarily to our modeling assumptions, the hitting time of distinct intervals are not  independent in general. In fact, hitting times of neighboring intervals should be quite close with positive probability in the limit of small intervals. However, for one part the asymptotic independence assumption should hold as long as intervals are not geographically close to one another, and from the other part, by linearity of the expectation, considering a model with correlated fragmentation times should not modify the value of $\E[\sum_{i\geq1} F_i(t)^{q+1}]$.
  \item[Choice of the fragmentation distribution.] The choice of a self-similar fragmentation procedure has been guided by the following heuristic picture. Consider the set of points visited by the Brownian motion $B^{\D}$ on $\partial \D$ after its first hitting time of a very small interval of length $m$. Using the Brownian caling, we think of these points as a scaled analogue of the set of positions in $\{0\}\times [0,1]$ hit by the Brownian motion $B^\H$ in the half-plane $\H = \R_+\times \R$, reflected at its boundary, started from a uniform point in $\{0\}\times [0,1]$. However, this process being recurrent, we need to introduce a cutoff time in order to define a proper fragmentation distribution. For example, we may set $\nu_A$ for the fragmentation corresponding to the hitting positions of $B^\H$ before its exit of $B(0,A)$. Note that $\nu_A$ converges to $(0,0,\cdots)$ as $A \to \infty$, but we might wish to correct our toy model by splitting an interval of size $\epsilon$ according to the distribution $\nu_{A(\epsilon)}$ with $A(\epsilon) \to \infty$ as $\epsilon \to 0$. This inhomogeneous fragmentation model might have an asymptotic behavior closer to the one observed in \cite{DBBM22}.
\end{description}

We end our discussion with the observation that in dimension greater than $2$, a similar analysis might be doable, but the cutoff mentioned above would no longer be needed as for a Brownian motion $W$ in the half-space $\R_+ \times \R^{d-1}$ with reflection at its boundary, the quantity
\[
  \mathcal{D}_d := \{B_t : t \geq 0, B_t \in \{0\} \times \R^{d-1} \}
\]
gives a well-defined set of points. We could therefore define a fragmentation process approaching the behavior of the convex hull of a Brownian motion in the ball of dimension $d$ with orthogonal reflection at its boundary.

For example, in dimension $3$, we write $\mathbb{B}$ for the ball of unit radius centered at $0$, and $B^\mathbb{B}$ for the Brownian motion in the ball with orthogonal reflection at its boundary. The fragmentation associated to the study of $\{B^\mathbb{B}_s, s \leq t\}$ would be a time-inhomogeneous process of triangles on the sphere. Each fragment would be a triangle $T$, and at a rate given by the asymptotic behavior of the probability for the Brownian motion to hit this triangle, specifically
\[
  \P(\tau > t) \underset{t \to \infty}{\sim} \exp\left( - c(T) t \right),
\]
the element would be divided into fragments given by the Delaunay triangulation of $T$ with vertices given by an i.i.d copy of $\mathcal{D}_3 \cap T$. Here $c(T)$ behaves proportionally to the area of $T$, see \cite{net3}. It would connect the growth of the area of the boundary of the convex hull of a Brownian motion in dimension $3$ to a self-similar fragmentation process with index $1$.

In this situation, writing $A_t$ for the area of the convex hull $C_t$ of $\{B^\mathbb{B}_s, 0 \leq s \leq t\}$, and using the same computations as in the introduction, we remark that
\[
  4 \pi^2 - A_t \approx \sum A(T)^2 
\]
where $A(T)$ represents the area of the triangle $T$, and the sum is taken over all fragments of the triangulation. We recall that for a self-similar fragmentation process $F$ with index $1$, there exists $C> 0$ such that
\[
   \sum_{i\geq 1} F_i(t)^2 \approx C/t,
\]
by \cite[Theorem 3]{BertoinAB03}.
It allows us to predict that $4\pi^2 - A_t$ should be of order $1/t$ as $t \to \infty$.

Note that Cauchy's surface area formula extends in higher dimension \cite{TsV}, allowing us the exact computation
\[
  \E[A_t] = 4 \E\left[ \lambda(\Pi C_t) \right],
\]
where $\lambda$ is the Lebesgue measure on $\R^2$ and $\Pi C_t$ is an orthogonal projection of $C_t$ onto $\R^2$, using again the invariance by rotation of the law of the Brownian motion. Estimating the difference between the volume of the convex subset of $\D$ by twice the difference of perimeters, we would obtain
\[
  4 \pi - \E[A_t] \approx \int_0^1 \P(\tau_x > t) \dd x \quad \text{as $t \to \infty$}
\]
where $\tau_x$ stands here for $\inf\{t> 0 : (1,0,0) \cdot B^\mathbb{B}_t > x\}$. Using again the narrow escape approximation, we can heuristically estimate this integral as
\[
 \int_0^1 e^{-c t (1 - x^2)} \dd x \underset{t \to \infty}\sim \frac{1}{2ct}
\]
which is consistent with our prediction.

\section*{Acknowledgements}
Our warmest thanks to V\'ictor Rivero for answering our questions on papers \cite{DR13} and \cite{DR16_err}.

\bibliographystyle{acm}
\bibliography{biblio.bib}

\begin{thebibliography}{10}

\bibitem{handbook}
{\sc Abramowitz, M., and Stegun, I.~A.}
\newblock Handbook of mathematical functions with formulas, graphs and
  mathematical tables.
\newblock Washington: {U}.{S}. {Department} of {Commerce}. xiv, 1964.

\bibitem{Berestycki03}
{\sc Berestycki, J.}
\newblock Multifractal spectra of fragmentation processes.
\newblock {\em Journal of Statistical Physics 113}, 3-4 (2003), 411--430.

\bibitem{BLevy}
{\sc Bertoin, J.}
\newblock {\em L{\'e}vy processes}, vol.~121 of {\em Camb. Tracts Math.}
\newblock Cambridge: Cambridge Univ. Press, 1998.

\bibitem{BertoinHom}
{\sc Bertoin, J.}
\newblock Homogeneous fragmentation processes.
\newblock {\em Probab. Theory Relat. Field. 121\/} (2001), 301--318.

\bibitem{BertoinSSF}
{\sc Bertoin, J.}
\newblock Self-similar fragmentations.
\newblock {\em Ann. Inst. Henri Poincar\'e Probab. Stat. 38}, 3 (2002),
  319--340.

\bibitem{BertoinAB03}
{\sc Bertoin, J.}
\newblock The asymptotic behavior of fragmentation processes.
\newblock {\em J. Eur. Math. Soc. (JEMS) 5}, 4 (2003), 395--416.

\bibitem{BGT}
{\sc Bingham, N., Goldie, C., and Teugels, J.}
\newblock {\em Regular variation}, vol.~27 of {\em Encycl. Math. Appl.}
\newblock Cambridge University Press, Cambridge, 1987.

\bibitem{CLUB09}
{\sc Caballero, M.~E., Lambert, A., and Uribe~Bravo, G.}
\newblock Proof(s) of the {Lamperti} representation of continuous-state
  branching processes.
\newblock {\em Probab. Surv. 6\/} (2009), 62--89.

\bibitem{Cau32}
{\sc Cauchy, A.~L.}
\newblock {\em M{\'e}moire sur la rectification des courbes et la quadrature
  des surfaces courbes}.
\newblock Lith. de C. Mantoux, 1832.

\bibitem{DBBM22}
{\sc De~Bruyne, B., B{\'e}nichou, O., Majumdar, S.~N., and Schehr, G.}
\newblock Statistics of the maximum and the convex hull of a {Brownian} motion
  in confined geometries.
\newblock {\em J. Phys. A, Math. Theor. 55}, 14 (2022), 1--17.

\bibitem{DR13}
{\sc Doney, R., and Rivero, V.}
\newblock Asymptotic behaviour of first passage time distributions for
  {L{\'e}vy} processes.
\newblock {\em Probab. Theory Relat. Fields 157}, 1-2 (2013), 1--45.

\bibitem{DR16_err}
{\sc Doney, R., and Rivero, V.}
\newblock Erratum to: ``{Asymptotic} behaviour of first passage time
  distributions for {L{\'e}vy} processes''.
\newblock {\em Probab. Theory Relat. Fields 164}, 3-4 (2016), 1079--1083.

\bibitem{GH1}
{\sc Goldschmidt, C., and Haas, B.}
\newblock Behavior near the extinction time in self-similar fragmentations {I}:
  {T}he stable case.
\newblock {\em Ann. Inst. Henri Poincar\'e Probab. Stat. 46}, 2 (2010),
  338--368.

\bibitem{GH2}
{\sc Goldschmidt, C., and Haas, B.}
\newblock Behavior near the extinction time in self-similar fragmentations.
  {II}: {Finite} dislocation measures.
\newblock {\em Ann. Probab. 44}, 1 (2016), 739--805.

\bibitem{H03}
{\sc Haas, B.}
\newblock Loss of mass in deterministic and random fragmentations.
\newblock {\em Stochastic Process. Appl. 106}, 2 (2003), 245--277.

\bibitem{H23tail}
{\sc Haas, B.}
\newblock Tail asymptotics for extinction times of self-similar fragmentations.
\newblock {\em Ann. Inst. Henri Poincar\'e Probab. Stat. 59}, 3 (2023),
  1722--1743.

\bibitem{Krell08}
{\sc Krell, N.}
\newblock Multifractal spectra and precise rates of decay in homogeneous
  fragmentations.
\newblock {\em Stochastic processes and their applications 118}, 6 (2008),
  897--916.

\bibitem{KypFluctu}
{\sc Kyprianou, A.~E.}
\newblock {\em Fluctuations of {L{\'e}vy} processes with applications.
  {Introductory} lectures}, 2nd ed.~ed.
\newblock Universitext. Berlin: Springer, 2014.

\bibitem{Lamp67bis}
{\sc Lamperti, J.}
\newblock Continuous state branching process.
\newblock {\em Bull. Am. Math. Soc. 73\/} (1967), 382--386.

\bibitem{Lamp67}
{\sc Lamperti, J.}
\newblock On random time substitutions and the {F}eller property.
\newblock In {\em Markov Processes and Potential Theory (Proc Sympos. Math.
  Res. Center, Madison, Wis., 1967)\/} (1967), 87--101.

\bibitem{Let}
{\sc Letac, G.}
\newblock An explicit calculation of the mean of the perimeter of the convex
  hull of a plane random walk.
\newblock {\em J. Theor. Probab. 6}, 2 (1993), 385--387.

\bibitem{Rupprecht2014}
{\sc Rupprecht, J.-F., B{\'{e}}nichou, O., Grebenkov, D.~S., and Voituriez, R.}
\newblock Exit time distribution in spherically symmetric two-dimensional
  domains.
\newblock {\em J. Stat. Phys. 158}, 1 (2014), 192--230.

\bibitem{narrowEscapeProblem}
{\sc Schuss, Z., Singer, A., and Holcman, D.}
\newblock The narrow escape problem for diffusion in cellular microdomains.
\newblock {\em Proceedings of the National Academy of Sciences 104}, 41 (2007),
  16098--16103.

\bibitem{net3}
{\sc Singer, A., Schuss, Z., Holcman, D., and Eisenberg, R.~S.}
\newblock Narrow escape. {I}.
\newblock {\em J. Stat. Phys. 122}, 3 (2006), 437--463.

\bibitem{Temme90}
{\sc Temme, N.~M.}
\newblock Uniform asymptotic expansions of a class of integrals in terms of
  modified {Bessel} functions, with application to confluent hypergeometric
  functions.
\newblock {\em SIAM J. Math. Anal. 21}, 1 (1990), 241--261.

\bibitem{TsV}
{\sc Tsukerman, E., and Veomett, E.}
\newblock Brunn-{Minkowski} theory and {Cauchy}'s surface area formula.
\newblock {\em Am. Math. Mon. 124}, 10 (2017), 922--929.

\end{thebibliography}

\end{document}